\documentclass[12pt]{article}    
\usepackage[margin=1.1in]{geometry}
\usepackage{graphicx}
\usepackage{amsmath,amsfonts,amsthm,mathrsfs,amssymb,cite}
\usepackage{epstopdf}
\usepackage{hyperref}
\usepackage[shortlabels]{enumitem}
\usepackage{soul}
\renewcommand{\hl}[1]{#1}
\numberwithin{equation}{section}

\newtheorem{thm}{Theorem}
\newtheorem{cor}{Corollary} 
\newtheorem{lem}{Lemma}[section]
\newtheorem{prop}{Proposition}

\theoremstyle{definition}
\newtheorem{defn}{Definition}[section]

\theoremstyle{remark}
\newtheorem{rem}{Remark}[section]

\usepackage{xcolor}

\newcommand{\pare}[1]{\left(#1\right)}

\newcommand{\abs}[1]{\left\lvert #1 \right\rvert}
\newcommand{\nrm}[1]{\left\|#1\right\|}

\newcommand{\AnD}{\quad\text{and}\quad}
\newcommand{\AND}{\qquad\text{and}\qquad}

\newcommand{\RR}{\mathbb{R}}
\newcommand{\Ss}{\mathbb{S}}
\newcommand{\CC}{\mathbb{C}}
\newcommand{\im}{\mathtt{i}}

\newcommand{\In}{\mathrm{in}}
\newcommand{\Sc}{\mathrm{sc}}

\newcommand{\sr}{r^\alpha}

\newcommand{\esdr}{e^{-\delta\sr/h}}
\DeclareMathOperator{\supp}{\mathrm{supp}}

\newcommand{\db}{\overline{\partial}}
\newcommand{\dd}{\partial}
\newcommand{\idb}{\mathscr{T}}
\newcommand{\idd}{\,\,\overline{\!\!\idb}}
\newcommand{\adst}{\mathcal{D}'}
\newcommand{\So}{\mathcal{S}}
\newcommand{\Sa}{\mathcal{S}_0}
\newcommand{\Sb}{\mathcal{S}_1}
\newcommand{\weg}{\mathcal{K}}
\newcommand{\imP}{\varphi}
\newcommand{\epa}{a}
\newcommand{\jump}{c}

\begin{document} 
\title{A New Type of CGO Solutions and its Applications in Corner Scattering}
\author{Jingni Xiao\footnote{Department of Mathematics, Rutgers University, Piscataway, NJ 08854-8019, USA,
 email:  jingni.xiao@rutgers.edu }}
\date{}

\maketitle
	
\begin{abstract}
	We consider corner scattering for the operator $\nabla \cdot \gamma(x)\nabla +k^2\rho(x)$ in $\mathbb{R}^2$, with $\gamma$ a positive definite symmetric matrix and $\rho$ a positive scalar function. A corner is referred to one that is on the boundary of the (compact) support of $\gamma(x)-I$ or $\rho(x)-1$, where $I$ stands for the identity matrix. \hl{We assume that $\gamma$ is a scalar function in a small neighborhood of the corner.} We show that any admissible incident field will be scattered by such corners, which are allowed to be concave. Moreover, we provide a brief discussion on the existence of non-scattering waves when $\gamma-I$ has a jump across the corner.
		In order to prove the results, we construct a new type of complex geometric optics (CGO) solutions.
\end{abstract}

\section{Introduction}
In this paper, we consider the question whether corners scatter. This is a natural question arising in both scattering and inverse scattering. It is also connected with the theory of transmission eigenvalue problems.
We start with the mathematical formulation of the scattering problem under consideration.
Let $\rho\in L^{\infty}(\RR^2;\RR)$ and $\gamma=(\gamma_{jl})\in L^{\infty}(\RR^2;\mathcal{M}_{2\times 2}(\RR))$ with $\rho-1$ and $\gamma-I_{2\times 2}$ compactly supported. We consider the following time-harmonic scattering problem 
\begin{equation}\label{eq:MainGov1}
	\begin{split}
		\nabla\cdot \gamma \nabla u + k^2 \rho u=0&\quad\mbox{in $~\RR^2$},
		\\u=u^{\In}+u^{\Sc}&\quad\mbox{in $~\RR^2$},
		\\\lim_{|x|\to\infty} |x|^{1/2}(\frac{\partial }{\partial |x|}u^{\Sc}-\im ku^{\Sc})=0, &~\mbox{ uniformly for all $\hat{x}:=\frac{x}{|x|}\in \Ss^1$}.
	\end{split}
\end{equation}
Here, $k\in\RR_+$ is the wavenumber; $u^{\In}$ is a given incident field which satisfies the Helmholtz equation
$\Delta u^{\In}+k^2u^{\In}=0$ in $\RR^2$; and $u^{\Sc}$ and $u$ are referred to as, respectively, the scattered field and the total field. 
The coefficients $\gamma$ and $\rho$ in \eqref{eq:MainGov1} represent the material properties of the media; in particular, $\rho-1$ and $\gamma-I_{2\times 2}$ stand for the inhomogeneities in a homogeneous background. 
Moreover, we assume that {$\rho$ and $\gamma$ satisfy} 
\begin{equation}\label{eq:gammEllip}
	\rho\ge a_0\quad\text{and}\quad
	\xi^T\gamma\,\xi\geq a_0
	\quad \mbox{{a.e. in $\RR^2$ for all $\xi\in\Ss^1$}},
\end{equation}
with some {positive} constant $a_0$. 
The last equation in \eqref{eq:MainGov1} is called the Sommerfeld radiation condition, which guarantees the out-going property of $u^{\Sc}$. 
As a consequence, the scattered field satisfies the following asymptotic expansion
\begin{equation*}
	u^{\Sc}(x)=\frac{e^{\im k|x|}}{|x|^{1/{2}}}\pare{u^\infty(\hat x)+O (\frac{1}{|x|} ) }\qquad \text{as}~|x|\to\infty,
\end{equation*}
uniformly for all $\hat{x}\in\Ss^1$ (see, \cite{ColKre19book}). The function $u^{\infty}$ defined on $\Ss^1$ is referred to as the far-field pattern. Moreover, by Rellich's lemma, $u^{\infty}$ is \hl{a} one-to-one correspondence with $u^{\Sc}$ restricted to $\RR^2\setminus\overline{\Omega}$, where $\Omega$ can be any open bounded simply-connected Lipschitz domain such that $\supp(\rho-1)\cup\supp(\gamma-I)\subseteq \overline{\Omega}$.  

We are interested in the situation when an incident field $u^{\In}$ with the associated wavenumber $k$ is not scattered as measured by an exterior or far-field observer, that is, when $u^{\infty}\equiv 0$. 
If this is the case, by Rellich's lemma, the corresponding scattered field $u^{\Sc}$ will be identically zero in $\RR^2\setminus\overline{\Omega}$, where $\Omega$ is a domain as specified preciously. As a consequence, we obtain a nontrivial eigenfunction for the so-called interior transmission eigenvalue problem on $\Omega$. Hence, we can conclude that, (the square of) a ``non-scattering wavenumber'' must be an interior transmission eigenvalue, where the non-scattering incident field along with the total field (restricted in $\Omega$) is a corresponding eigenfunction. However, the converse is not valid. In fact, if a transmission eigenfunction is also an incident wave, then it needs also to ``exist'' in the exterior of $\Omega$ and remaining a solution to the corresponding PDE. In particular, the eigenfunction has to be sufficiently smooth (and in fact real-analytic) at least up to the boundary of $\Omega$. Unfortunately, this required smoothness or extension is not guaranteed for general interior transmission eigenfunctions. For the existence and regularity of interior transmission eigenfunctions, we refer the reader to the monograph \cite{CCH16} and the references therein. 

It is known from existing results that the possibility of non-scattering, or equivalently, the ``extend-ability'' of interior transmission eigenfunction, depends on the shape and structure of the media (the coefficients of the PDEs).
In particular, for spherically stratified media the existence of non-scattering waves is known; See, \cite{ColtonMonk88} for the case of $\gamma\equiv 1$ and $\rho=\rho(|x|)$, and similar arguments work when $\gamma=\gamma(|x|)$ which is a scalar function. In this case, the set of eigenfunctions is identical to that of non-scattering waves, and hence there are infinitely many non-scattering wave numbers.
On the other hand, it has been proven {recently in} \cite{VogX20,VogX21} that for any $C^{2,\alpha}$ strictly convex domain in $\RR^2$ that is not a sphere, there are at most finitely many non-scattering wave numbers concerning incident plane waves with a fixed incident direction. Similar finiteness results are also proven for ellipses and its small perturbations concerning Herglotz incident waves with a fixed density function.

In the current paper, we investigate the case when there is a corner on the boundary of {$\supp(\rho-1)\cup\supp(\gamma-I)$. 
The} first related result is established in \cite{BPS14} for the operator $\Delta+k^2\rho$ in $\RR^n$, $n\ge 2$. The authors prove that, when the support of $\rho-1$ contains a (convex) rectangular corner on the boundary, at which $\rho\neq 1$, then all incident fields must be scattered. 
One of the main technique in the argument is the application of the so-called complex geometric optics (CGO) solutions, which has a leading term $e^{\im \tau x\cdot \eta}$ with $\eta\in\CC^n$ and $\eta\cdot\eta=0$ (or $k^2/\tau^2$) for an arbitrary parameter $\tau \gg 1$. The construction of CGO solutions dates back to \cite{SyU87} and has extensive developments and applications in inverse problems, which we choose not to expand here and refer the reader to \cite{Uhl14} and the references therein.
One of the advantages of such solutions while applied to corner scattering, is the exponential decaying property of the term $e^{\im \tau x\cdot \eta}$ as $\tau\to\infty$, provided that $\Im \eta\cdot x>\delta>0$ for some constant $\delta$.
Since \cite{BPS14}, there have been several developments on corner scattering concerning the operator $\Delta+k^2\rho$ with $\rho\neq 1$ at the corner. 
The ``always-scattering'' results for convex corner in $\RR^2$ and convex conic corner in $\RR^3$ are established in \cite{PSV17}. Using a different approach, it is {proven in} \cite{ElH18} that concave conic corners in $\RR^2$ and $\RR^3$ and edges in $\RR^3$ also lead to always-scattering. The corresponding results for the source problem concerning $\Delta+k^2$ is considered in \cite{Bla18}. In contrast to the medium scattering problems, harmonic functions, instead of a full CGO solution for $\Delta+q(x)$, suffice for the case of source scattering. A quantitative version of the always scattering result for high curvature points on the boundary is established in \cite{BlaLiu20arX}. Related results concerning different models including the same operator but in the space $\mathbb{H}^n$, the same operator with conductive boundary conditions, similar problems for Maxwell equations and for the Lam\'e system, can be found in \cite{BlaV20,DCH21,LiX17Corner,BLX19arXiv,BlaLin19}, etc.
In \cite{CaX21}, for the first time corner scattering for the operator $\nabla \cdot \gamma\nabla +k^2\rho$ is considered, and corner scattering results are shown for convex corners in $\RR^2$.

In this paper, we generalize the results in \cite{CaX21} and consider corners which can be concave. In order to establish the results for concave corners, we adopt CGO solutions with leading term of the form $e^{z^\alpha/h}$, where $z=x_1+\im x_2$ and $\alpha\in(0,1)$. This type of CGO solutions was suggested by E. Bl{\aa}sten in \cite{Bla18}. However, since the (essential) results in \cite{Bla18} are concerned with the source problem, $e^{z^\alpha/h}$ is solely used as a harmonic function, which suffices for the source problem as mentioned before. Full CGO solutions are not needed or constructed in \cite{Bla18}. 
We apply the d-bar method for the construction of CGO solutions, for the case considered in the current {paper.} 

The rest of the paper is structured as follows. In Section~\ref{sec:2} we represent and discuss our main result, namely Theorem~\ref{thm:main}, and sketch its proof. It is seen from the proof that the essence of Theorem~\ref{thm:main} is Proposition~\ref{prop:mainLocal}, which is associated with a local problem around the corner.  
In Section~\ref{cgo}, CGO solutions to $\nabla \cdot \gamma\nabla w+k^2\rho w=0$ with leading term of the form $e^{z^\alpha/h}$ are constructed. The corresponding $L^{p}$ estimates are established.
Applying the CGO solution, we analyze in Section~\ref{sec:corner} an integral locally near the corner, and eventually give the proof of Proposition~\ref{prop:mainLocal}.

\section{The Main Results}\label{sec:2}
In this section, we shall present our main results. {We first} 
introduce some notations and definitions.

Assume throughout the paper that $\rho\in L^{\infty}(\RR^2;\RR)$ and $\gamma\in L^{\infty}(\RR^2;\mathcal{M}_{2\times 2}(\RR))$ with \eqref{eq:gammEllip} satisfied and $\rho-1$ and  $\gamma-I$ {compactly} 
supported.
Denote $D$ as the bounded open domain such that $\RR^2\setminus\overline{D}$ is the unbounded connected component of $\RR^2\setminus (\supp (\rho-1) \cup  \supp (\gamma-1))$. {In other words, the domain $D$ represents the shape and location of the inhomogeneous scatterer.}
\begin{defn}
	Given $u^{\In}$ an incident field, which satisfies $\Delta u^{\In}+k^2u^{\In}=0$ in $\RR^2$, let $u^{\Sc}=u-u^{\In}$ be the solution to the scattering problem \eqref{eq:MainGov1} and $u^{\infty}$ the corresponding far-field pattern. 
	We say that {$(u^{\In};\gamma,\rho)$ \emph{scatters} if $u^{\infty}$ is not identically zero; otherwise it is  referred to as a pair of \emph{non-scattering wave and medium}}
\end{defn}

For $\theta_0\in(0,\pi)$, denote by $\weg_{\theta_0}$ the open circular sector of aperture $2\theta_0$ as $\weg_{\theta_0}=\{x=r(\cos\theta,\sin\theta)\in\RR^2;\,0<r<\infty,-\pi+\theta_0<\theta<\pi-\theta_0\}$. 
Given $R\in\RR_+$, $\weg_{\theta_0,R}$ is the bounded sector defined by $\weg_{\theta_0,R}=\weg_{\theta_0}\cap B_R$, where $B_R$ is the open ball in $\RR^2$ centered at the origin and of radius $R$. 

\begin{defn}\label{def:corner}
	We say that there is a \emph{corner} $x_0$ for 
	$(\gamma,\rho)$ if, up to a rigid change of coordinates, $x_0$ is the origin and {$D\cap B_R=\weg_{\theta_0,R}$} for some $R\in\RR_+$ and $\theta_0\in(0,\pi)\setminus\{\pi/2\}$. 
	We also refer to $2\theta_0$ as the aperture of the corner and $R$ its size. 
	
	Furthermore, $(\gamma,\rho)$ is said to be \emph{admissible} at the corner if the following two conditions are satisfied in $\weg_{\theta_0,\tilde{\epa}}$ for some constant $\tilde{\epa}\in(0,R)$. 
	First, $\gamma$ is a scalar function\footnote{That is, $\gamma(x)=\tilde{\gamma}(x)I_{2\times 2}$ with $\tilde{\gamma}$ scalar. We still write $\tilde{\gamma}$ as $\gamma$ for simplicity.} in $\weg_{\theta_0,\tilde{\epa}}$ and satisfies $\gamma\in W^{1,\infty}(\weg_{\theta_0,\tilde{\epa}})$ and $\Delta \gamma^{1/2}\in L^{\infty}(\weg_{\theta_0,\tilde{\epa}})$.
	Second, there are constants $c_j\in\RR$ and $C_j,\beta_j\in\RR_+$ such that
	\begin{equation*}
		\abs{\gamma^{-1/2}\pare{\gamma-1}-c_1}\le C_1 |x|^{\beta_1}
		\AnD
		\abs{\gamma^{-1/2}\pare{\rho-1}-c_2}\le C_2 |x|^{\beta_2}
		\quad\mbox{a.e. in $\weg_{\theta_0,\tilde{\epa}}$.}
	\end{equation*}
	In this case, $c_1$ is referred to as the \emph{essential jump} of $\gamma$ at the corner and $c_2$ as that of $\rho$.
	Moreover, if $c_1\neq 0$ we say that $\gamma$ has a (non-trivial) jump at the corner $x_0$; otherwise if $c_1=0$ then there is no jump for $\gamma$, in which case $\beta_1$ is called the {\emph{order of vanishing} for $\gamma-1$} 
	at the corner; and the same terminology applies to $\beta_2$ for {$\rho-1$}.
\end{defn}

The following is the main result of the current paper. It provides the ``corner always scattering'' results in two different situations {depending on whether} the corner appears as a jump of $\gamma$ or that of $\rho$. In the special case when $\rho$ has a jump at the corner whereas {$\gamma-1$ is of vanishing order $2+\beta$ with some $\beta>0$,} 
we show that all incident waves will be scattered{; see, Item (I-a) in the statement of Theorem }\ref{thm:main}. 
{In particular, it generalizes some existing results considering the operator $\Delta+k^2\rho$; see for instance,} \cite{BPS14,PSV17,ElH18}. 
{In the other} cases, the always scattering results are established for incident waves with certain behavior at the corner which may involve, for instance, the (non-)vanishing order of the incident wave and/or its gradient. 
{Moreover, some of these conditions can be shown necessary for the ``always-scattering'' conclusion. We discuss this with more details at the end of this section.}

\begin{thm}\label{thm:main}
	Let $\rho$ and $\gamma$ be as specified at the beginning of this section. Suppose that there is a corner $x_0$ of aperture $2\theta_0$ for $(\gamma,\rho)$ where the latter is admissible. Given an incident field $u^{\In}$ with a wave number $k\in\RR_+$, $(u^{\In};\gamma,\rho)$ is always scattering provided that
	\begin{enumerate}[(I)]
		\item $\rho$ has a jump at $x_0$, while $\gamma$ does not {and}  
		\begin{enumerate}
			\item $\beta_1>2$, {where $\beta_1$ is the vanishing order of $\gamma-1$,} or,
			\item $\beta_1>1$, with $u^{\In}(x_0)\neq 0$, or,
			\item $\beta_1>0$, with $u^{\In}(x_0)\neq 0$ and $\nabla u^{\In}(x_0)= 0$, or,
		\end{enumerate}
		\item $\gamma$ has a jump at $x_0$ and 
		\begin{enumerate}
			\item $u^{\In}(x_0)\neq 0$ and $\nabla u^{\In}(x_0)\neq 0$, 
			or,
			\item $u^{\In}(x_0)=0$, with 
			\\$2\theta_0\neq {l\pi}/{N_0}$  for some $l\in\mathbb{N}_+$ and $N_0$  the vanishing order of $u^{\In}$ at $x_0$, (for example, $2\theta_0\in(0,2\pi)\setminus (\mathbb{Q}\pi)$,)
			or,
			\item $u^{\In}(x_0)\neq 0$ and $\nabla u^{\In}(x_0)= 0$, with 
			\begin{enumerate}
				\item when $2\theta_0\in\{\pi/2,3\pi/2\}$, assuming $c_1\neq c_2$ where $c_1$ and $c_2$ are the essential jumps for, respectively, $\gamma$ and $\rho$ as in Definition~\ref{def:corner}, or,
				\item if $2\theta_0\notin\{\pi/2,3\pi/2\}$, after the change of coordinates as in Definition~\ref{def:corner}, 
			\end{enumerate}
		\begin{equation*}					
			u^{\In}(x)\neq v_0J_0(kr)+c_0v_0J_2(kr)(e^{\im 2\theta}+e^{-\im 2\theta})+\sum_{m=2}^{\infty}J_{2m}(kr)\pare{b_{m,1}e^{\im 2m\theta}+b_{m,2}e^{-\im2m\theta}}
		\end{equation*}
	\begin{enumerate}
		\item[]  	near the corner, where $v_0$ is a non-zero constant,  $c_0=(1-c_2/c_1)/\cos 2\theta_0$ and $J_{m}$ is the Bessel function of the first kind.
	\end{enumerate}
		\end{enumerate}
	\end{enumerate}
\end{thm}
{Before the proof of Theorem}~\ref{thm:main}, we first state the existence and uniqueness of solutions in $ H^1_{loc}(\RR^2)$ to the scattering problem \eqref{eq:MainGov1}. 
{We refer the readers to} \cite[Section 1.4]{CCH16} for a proof of this result. Note that in order to apply the arguments as in \cite{CCH16} for the coefficient $\gamma$ with $L^{\infty}$ regularity, we need to use the unique continuation principle from \cite{Ale12UCP}.
\begin{lem}
{Let $\rho\in L^{\infty}(\RR^2;\RR)$ and $\gamma\in L^{\infty}(\RR^2;\mathcal{M}_{2\times 2}(\RR))$ satisfy}  \eqref{eq:gammEllip} with $\rho-1$ and  $\gamma-I$ compactly supported.
Given an incident field $u^{\In}$, there is a unique solution $u\in H^1_{loc}(\RR^2)$ to \eqref{eq:MainGov1}.
\end{lem}
\begin{proof}[Proof of Theorem~\ref{thm:main} (Sketch)]
Theorem~\ref{thm:main} is proven by contradiction. Assume that an incident field $u^{\In}$ with wavenumber $k\in\RR_+$ is not scattered, that is, $u^{\infty}\equiv 0$. 
Then by Rellich's lemma, $u^{\Sc}\equiv 0$ in $\RR^2\setminus\overline{\Omega}$, where $\Omega$ is any open bounded Lipschitz domain such that $\supp(\rho-1)\cup\supp(\gamma-I)\subseteq \overline{\Omega}$ and that $\RR^2\setminus\overline{\Omega}$ is connected. 
Then we obtain the following interior transmission eigenvalue problem 
\begin{equation}\label{eq:ITEP}
	\begin{split}
		\nabla\cdot \gamma \nabla u + k^2 \rho u=0,\quad  \Delta u^{\In} + k^2 u^{\In}=0,&\qquad\mbox{in $\Omega$},\\
		u=u^{\In},\quad  \partial_{\nu}^{\gamma} u=\partial_{\nu}u^{\In},&\qquad\mbox{on $\partial\Omega$},
	\end{split}
\end{equation}
where $\nu$ is the outward normal vector of $\partial \Omega$ and $\partial_{\nu}^{\gamma}=\nu^T\gamma \nabla$.
We shall take such a domain $\Omega$ which satisfies that, up to a rigid change of coordinates,  
$\Omega\cap B_{\tilde{\epa}}=\weg_{\theta_0,\tilde{\epa}}=D\cap B_{\tilde{\epa}}$. 
In particular, $u$ and $v:=u^{\In}$ satisfies the following local problem at the corner
\begin{equation}\label{eq:ITEPlocal}
	\begin{split}
		\nabla\cdot \gamma \nabla u + k^2 \rho u=0,\quad  \Delta v + k^2 v=0,&\qquad\mbox{in $\weg_{\theta_0,\tilde{\epa}}$},\\
		u=v,\quad  \partial_{\nu}^{\gamma} u=\partial_{\nu}v,&\qquad\mbox{on $\partial\weg_{\theta_0}\cap B_{\tilde{\epa}}$},
	\end{split}
\end{equation}
where $\gamma$ is a scalar function in $B_{\tilde{\epa}}$.
Theorem~\ref{thm:main} is then a direct consequence of the following proposition, which states that as a nontrivial solution to \eqref{eq:ITEPlocal}, $v$ can not be extended into $B_{\tilde{\epa}}$ as a solution to the Helmholtz equation, under the assumptions of Theorem~\ref{thm:main} with $v=u^{\In}$.
\begin{prop}\label{prop:mainLocal}
	Let $\rho$, $\gamma$ be the same as those in Theorem~\ref{thm:main} with a corner where $(\gamma,\rho)$ is admissible.
	Given $(u,v)$ a non-trivial solution to \eqref{eq:ITEPlocal} (after the change of coordinates as in Definition~\ref{def:corner}), 
	suppose {that $u\in H^1(\weg_{\theta_0,\tilde{\epa}})$ and that $v$ can} be extended into $B_{\tilde{\epa}}$, as a solution to $\Delta v+k^2v=0$. Then none of the conditions for $\rho$, $\gamma$, $\theta_0$ and $u^{\In}=u^{\In}|_{B_{a}}:=v$ in Items (I)-(II) from Theorem~\ref{thm:main} can be valid. 
\end{prop}
	
\end{proof}
The proof of {Proposition}~\ref{prop:mainLocal} is given in Section~\ref{sec:proofProp}. 
We would like to give some remarks on Theorem~\ref{thm:main} and {Proposition}~\ref{prop:mainLocal} for the remainder of this section.

Firstly, it is seen from Proposition~\ref{prop:mainLocal} that, Theorem~\ref{thm:main} is also valid for point source type of incident waves, as long as the source does not locate at the corner. In fact, all the conclusions hold for any ``local incident fields'' which satisfies the Helmholtz equation in a neighborhood of the corner{, providing that the solution is in $H^1$}.

{Secondly, applying Proposition}~\ref{prop:mainLocal} {we can generalize the result on the ``non-uniform approximation'' of interior transmission eigenfunctions as in} \cite{CaX21}.
More precisely, consider the interior transmission eigenvalue problem \eqref{eq:ITEP}. Assume that there is a (convex or concave) corner {on} the boundary of $\Omega$ where $\gamma$ and $\rho$ are admissible in the sense of Definition~\ref{def:corner}. Then we can conclude that, whenever an eigenfunction $v$ is approximated by the Herglotz functions
\begin{equation*}
	v_{\varepsilon}(x):=\int_{{\mathbb S}^{1}}g_{\varepsilon}(d) e^{ikx\cdot d}ds(d),\qquad g_{\varepsilon}\in L^2({\mathbb S}^{n-1})
\end{equation*}
as $\varepsilon\to 0+$ in the $H^1(\Omega)$ norm, we must have that 
$\lim\sup\|g_{\varepsilon}\|_{L^2({\mathbb S}^1)}=\infty$, where the limit is taken as $\varepsilon\to 0+$.

{Lastly}, the ``always scattering'' or the ``non-extendability'' results in Theorem~\ref{thm:main} and {Proposition}~\ref{prop:mainLocal} are established with certain types of waves or solutions excluded, especially when there is a jump for $\gamma$ at the corner. 
{In fact, some of these exclusions are necessary. In other words, we can construct non-scattering waves and media which belong to these excluded classes. More specifically, let us consider the special case when $\gamma=a_0=\rho$ in the medium $D$ with $a_0$ a positive constant different from $1$. 
	Assume that $u^{\In}$ is a non-scattering incident wave with the wavenumber $k$, and $u$ is the corresponding total wave. Then we observe that $w_D=u-u^{\In}$ is a Dirichlet eigenfunction in $D$ and $w_N=a_0u-u^{\In}$ is a Neumann eigenfunction corresponding to the same eigenvalue $k^2$.  Moreover, one of these two ``eigenfunctions'' can be ``degenerated''; that is, $w_D$ or $w_N$ is allowed to be identically zero. 
In this case, we can interpret the existence of non-scattering wave numbers as, whether there are such eigenfunctions with $a_0w_D-w_N$ extendable to the exterior of the corner and remain as a solution to the Helmholtz equation.
For example, if $D$ is a rectangle, say, $D=(0,\pi)^2$, then we can find infinite number of non-scattering wavenumbers $k=\sqrt{k_1^2+k_2^2}$ with $k_1,k_2\in\mathbb{N}$, where (some of) the corresponding non-scattering incident waves are $a_1\sin(k_1x_1)\sin(k_2x_2)+a_2\cos(k_1x_1)\cos(k_2x_2)$.
This shows the necessity of the extra conditions Item (II-c-i) and partly in Item (II-b), from Theorem} \ref{thm:main}, {for the ``always scattering'' conclusion. Nonscattering waves for corners with other apertures regarding Item (II-c-i) can be constructed by taking $w_N$ trivial. 
In a ongoing work, we expect to establish further the existence or construction of non-scattering waves and media in a more general setting.
} 

The rest of this paper is devoted to the proof of Proposition~\ref{prop:mainLocal}.

\section{Construction of CGO Solutions}\label{cgo}


Let $G$ be an open and simple connected Lipschitz domain that is a subset of (or equal to) $\weg_{\theta_1,R_1}$ for some $\theta_1\in(0,\pi)$ and $R_1\in\RR_+$. 
In this section, we shall construct solutions to
\begin{equation*}
\Delta W + qW=0\qquad \text{in} \quad G
\end{equation*}
which are of the form
\begin{equation}\label{eq:CGOform}
W=e^{-\Phi/h}  W_h
\qquad\text{with}
\quad
W_h(x)= A(x)+w_h(x)\AnD\Phi(x)=r^{\alpha}e^{\im\alpha\theta }.
\end{equation}
Hereinafter, the parameter $\alpha\in(0,1)$, $h$ is an arbitrary (and sufficiently small) positive constant, 
 and $(r,\theta)$ is the polar coordinates of $x$ with $\theta\in(-\pi,\pi)$.
 We shall prove that
\begin{prop}\label{prop:CGO1}
	Let $q\in L^{\infty}(G)$, $A\in L^{p_0}(G)$ with $A$ a holomorphic function in $G$ and $p_0$ a constant with $2/\alpha\le p_0\le\infty$. Then there is a constant $h_0\in(0,1)$ such that for any $h\in(0,h_0)$, there exists a solution $W$ to $\Delta W + qW=0$ in $G$ which is of the form \eqref{eq:CGOform}. Moreover, for any $p_1\in(1,p_0)$ and any $p_2\in (1,2/\alpha)$, we have
	\begin{equation}\label{eq:CGOresid}
	\|w_h\|_{L^{p_1}}\le C\text{\hl{$h$}}\|A\|_{L^{p_1}}
	\AND
	\|w_h\|_{W^{1,p_2}}\le C\|A\|_{L^{p_2}}
	\end{equation}
	with some constant $C$ independent of $h$ (depending on $p_1$, $p_2$, $h_0$, $\alpha$, $G$ and $\|q\|_{L^\infty}$).
\end{prop}
We postpone the proof of Proposition~\ref{prop:CGO1} to the end of this section. It is noted that the regularity assumption $q\in L^{\infty}(G)$ can be relaxed by using approximate cut-off arguments similar to those developed in \cite{Bukh08} and \cite{BTW20}. We choose not to present the more refined result in this paper to simplify the exposition.
As a direct consequence of Proposition~\ref{prop:CGO1}, we have
\begin{cor}\label{cor:CGO}
	Let $\alpha\in(0,1]$, $p_0\in[2/\alpha,\infty)$ and $A\in L^{p_0}(G)$ be the same as in Proposition~\ref{prop:CGO1}. 
	Let $\gamma$ and $\rho$ be functions defined on $G$ with $\gamma$ positive and $q:=k^2\rho\gamma^{-1}-\gamma^{-1/2}\Delta \gamma^{1/2}\in L^{\infty}(G)$.
	Then there is a constant $h_0\in(0,1)$ such that for any $h\in(0,h_0)$, there exists a solution $w$ to \begin{equation*}
	\nabla\cdot \gamma \nabla w + k^2 \rho w=0\quad\mbox{in $G$}
	\end{equation*}
	of the form 
	\begin{equation}\label{eq:CGO}
	w=\gamma^{-1/2}(A+w_h(x))e^{-\Phi/h}
	\qquad\text{with}
	\quad
	\Phi(x)=r^{\alpha}e^{\im\alpha\theta },~\theta\in(-\pi,\pi),
	\end{equation}
	where $w_h$ satisfies the same estimates as in Proposition~\ref{prop:CGO1}.
\end{cor}
\begin{proof}
	Let $W=\gamma^{1/2}w$. Then $W$ is of the form \eqref{eq:CGOform} and satisfies $\Delta W+qW=0$ with $q=k^2\rho\gamma^{-1}-\gamma^{-1/2}\Delta \gamma^{1/2}$. 
	The proof is then complete by applying Proposition~\ref{prop:CGO1}.
\end{proof} 
For the proof of Proposition~\ref{prop:CGO1}, we first identify the complex domain $\CC$ with the real domain $\RR^2$ by $z=x_1+\im x_2\in\CC$ for $x=(x_1,x_2)\in\RR^2$.
We shall make use of the $\db$ and $\dd$ operators defined by $\db=\frac{1}{2}\pare{\partial_{x_1}+\im \partial_{x_2}}$ and $\dd=\frac{1}{2}\pare{\partial_{x_1}-\im \partial_{x_2}}$.
Then $\Delta=4\db\dd=4\dd\db$.
Notice that 
\begin{equation*}
	\Phi(z)=|z|^\alpha \pare{\cos(\alpha\arg z)+\im\sin(\alpha\arg z)},
\end{equation*}
which is (a branch of $z^{\alpha}$ and) uniquely defined 
and holomorphic in $\CC\setminus\{[0,\infty)\}$ where $\theta=\arg z$ is taken in $(-\pi,\pi)$
; see also, \cite{Bla18}.

Let $\idb$ and $\idd$ be the Cauchy operators defined by
\begin{equation*}
	\idb f=\idb_G f=\frac{1}{\pi}\int_{G}\frac{f(\zeta)}{z-\zeta}d\vec{\zeta}\AND \idd f=\idd_G f=\frac{1}{\pi}\int_{G}\frac{f(\zeta)}{\overline{z}-\zeta}d\vec{\zeta},
\end{equation*}
for any $f\in L^1(G)$ (or in general for any distribution $f\in \adst(G)$), 
where $d\vec{\zeta}$ is understood as $dx_1dx_2$ with $\zeta=x_1+\im x_2$.
Then $\idb$ is a right inverse of $\db$, and $\idd$ is a right inverse of $\dd$, in the sense that (see, \cite{Vekua62}) 
\begin{equation}\label{eq:dbinv}
	\db\idb f = f\AND 
	\dd\idd  f= f,
	\qquad
	f\in \adst(G).
\end{equation}
The following mapping properties are known (see e.g., \cite{Vekua62,BTW20}).
\begin{lem}\label{lem:dbinvMap}
	The operators $\idb$ and $\idd$ are bounded from $L^p(G)$ to $W^{1,p}(G)$ with any $p\in(1,\infty)$.
\end{lem}

Denote 
\begin{equation*}
	\imP=\Im\Phi=r^\alpha \sin(\alpha\theta),\qquad \theta\in (-\pi,\pi).
\end{equation*}
By a straightforward computation we have
\begin{equation}\label{eq:RePinv}
	(\db\imP)^{-1} =\im \frac{2}{\alpha}\, r^{1-\alpha}e^{-\im(1-\alpha)\theta}
	,\qquad \theta\in(-\pi,\pi),
\end{equation}
which is anti-holomorphic in $G$.
For $f\in \adst(G)$, define the operator
\begin{equation*}
\So f=\So_{h,G,q}f=\im\frac{h}{4}\Sb \Sa (qf),
\end{equation*}
where
\begin{equation*}
	\Sa f={\Sa}_{,h,G}f=\idd \pare{e^{-2\im\imP/h}f},
\end{equation*}
and
\begin{equation*}
	\Sb f ={\Sb}_{,h,G} f = e^{2\im\imP/h}\frac{f}{2\db\imP}
	-\idb (e^{2\im\imP/h}\db(\frac{ f}{2\db\imP})).
\end{equation*}
Then we have the following two lemmas for $\So$.
\begin{lem}\label{lem:SoSol}
	Let $w=\So f$. Then $w$ solves 
	\begin{equation*}
		4 e^{2\im\imP/h}\dd\pare{e^{-2\im\imP/h}\db w} = -qf\qquad\text{in}\quad G.
	\end{equation*}
\end{lem}
\begin{proof}
	Applying \eqref{eq:dbinv} we deduce that
	\begin{equation*}
		\frac{4}{\im h}\db\So f=\db\Sb f_0
		= \db(e^{2\im\imP/h}\frac{f_0}{2\db\imP})- e^{2\im\imP/h}\db\frac{ f_0}{2\db\imP}
		=\frac{\im}{h} e^{2\im\imP/h} f_0,
	\end{equation*}
with $f_0=\Sa (qf)$.
As a consequence,
\begin{equation*}
	-4\dd e^{-2\im\imP/h}\db\So f
	=  \dd f_0
	=\dd \idd \pare{e^{-2\im\imP/h}f}
	=  e^{-2\im\imP/h}qf,
\end{equation*}
which completes the proof.
\end{proof}
\begin{lem}\label{lem:Smapping}
	If $q\in L^{\infty}(G)$ (and $\alpha\in(0,1)$), then for $h$ sufficiently small we have
	\begin{equation*}
		\|\So f\|_{L^{p}(G)}
		\le C_ph\|q\|_{L^{\infty}{(G)}}\|f\|_{L^p(G)},\qquad \text{for all } ~f\in L^p(G),~ 1<p<\infty,
	\end{equation*}
and
\begin{equation*}
	\|\So f\|_{W^{1,p}(G)}
	\le C_p\|q\|_{L^{\infty}{(G)}}\|f\|_{L^p(G)},\qquad f\in L^p(G),~1<p<2/\alpha,
\end{equation*}
where $C_p$ is a constant independent of $f$ and $h$.
\end{lem}
\begin{proof}
	It is seen directly from Lemma~\ref{lem:dbinvMap} that 
	\begin{equation*}
		\|\Sa f\|_{W^{1,p}(G)}\le C\|f\|_{L^p(G)},\qquad 1<p<\infty,
	\end{equation*}	
where $C$ is some constant independent of $f$ and $h$ (and the same is valid for the rest of the proof but with different values of $C$).
	Recall from \eqref{eq:RePinv} that $(\db\imP)^{-1}\in L^{\infty}(G)$,
	\begin{equation*}
	\dd \frac{1}{\db\imP}=0\AND
		\db \frac{1}{\db\imP} =2\im\frac{1-\alpha}{\alpha}r^{-\alpha}e^{\im \alpha \theta}
	,\qquad \theta\in(-\pi,\pi)	.
	\end{equation*}
Hence
	\begin{equation*}
		\|\db (2\db\imP)^{-1}\|_{L^p(\weg_{\theta_1,R_1})}
		=\frac{1-\alpha}{\alpha}\pare{\frac{2\theta_1}{2-\alpha p}}^{1/p} R_1^{2/p-\alpha},\qquad p<2/\alpha.
	\end{equation*}

	Applying Sobolev inequalities and Lemma~\ref{lem:dbinvMap} {we have}
	\begin{equation*}
		\nrm{\idb (e^{2\im\imP/h}\db\frac{ f}{2\db\imP})}_{L^{p}(G)}
		\le C\nrm{\idb (e^{2\im\imP/h}\db\frac{ f}{2\db\imP})}_{W^{1,p_1}(G)}
		\le C
		\nrm{\db\frac{ f}{2\db\imP}}_{L^{p_1}(G)}
	\end{equation*}
with $1/p_1\le1/p+1/2$ and $1<p_1<\infty$.
By H\"{o}lder's inequality we obtain
\begin{equation*}
	\|f\db (2\db\imP)^{-1}\|_{L^{p_1}(G)}
	\le \|f\|_{L^{p_2}(G)}\|(2\db\imP)^{-1}\|_{L^{p_2^*}(G)},
\end{equation*}
where  $1/p_2+1/p_2^*\le1/p_1$ and $p_2^*<2/\alpha$ (and hence $1<p_1\le p_2^*<2/\alpha$).
{Therefore}, 
	\begin{equation*}
		\begin{split}
			\nrm{\idb (e^{2\im\imP/h}\db\frac{ f}{2\db\imP})}_{L^{p}(G)}
		&\le C
		(\|\db f\|_{L^{p_1}(G)}+
		\|f\db (2\db\imP)^{-1}\|_{L^{p_1}(G)})
		\\&\le C(\|f\|_{W^{1,p_1}(G)}+
		\|f\|_{W^{1,p_3}(G)})
		\le C\|f\|_{W^{1,p}(G)},
		\end{split}
	\end{equation*}
for $1/p_3\le1/p_2+1/2$ and $p\ge p_j$, $j=1,3$.
In fact, we can choose $p_2^*$ and $p_j$, $j=1,2,3$, {such that} 
\begin{equation*}
	\begin{split}
		\nrm{\idb (e^{2\im\imP/h}\db\frac{ f}{2\db\imP})}_{L^{p}(G)}
		\le C_p\|f\|_{W^{1,p}(G)}\qquad\mbox{for any $p\in(1,\infty)$}.
	\end{split}
\end{equation*}
For example, when $p\ge2/\alpha$, one can take $p_1=p_3=p_2^*=2$ and $p_2=\infty$; for $p\in[2,2/\alpha)$, we can choose $p_1=p_2^*=p$, $p_3=2$ and $p_2=\infty$; as for $p\in(1,2)$, one can specify $p_1=p_3=p$, $p_2^*=2$ and $1/p_2=1/p-1/2$.
	Hence for $1<p<\infty$ {we have}
	\begin{equation*}
		\|\Sb f\|_{L^{p}(G)}
		\le\nrm{\frac{f}{2\db\imP}}_{L^{p}(G)}
		+\nrm{\idb (e^{2\im\imP/h}\db(\frac{ f}{2\db\imP}))}_{L^{p}(G)}
		\le C
		\|f\|_{W^{1,p}(G)}.
	\end{equation*}
Thus by Lemma~\ref{lem:dbinvMap} again,
\begin{equation*}
		\|\So f\|_{L^{p}(G)}
		\le C h
		\|\Sa (qf)\|_{W^{1,p}(G)}
		\le Ch \|qf\|_{L^{p}(G)}
		\le Ch\|q\|_{L^{\infty}{(G)}}\|f\|_{L^p(G)}.
	\end{equation*}

We are left to estimate $\So f$ in $W^{1,p}(G)$.
Denote $f_1=e^{2\im\imP/h}(2\db\imP)^{-1} f$. Then
\begin{equation*}
	\dd f_1=e^{2\im\imP/h}\pare{\frac{\im}{h}\frac{\dd \imP}{\db\imP}f+\frac{\dd f}{2\db\imP}}
	\AND
	\db f_1=e^{2\im\imP/h}\pare{\frac{\im}{h}f+f\db \frac{1}{2\db\imP}+ \frac{\db f}{2\db\imP}}.
\end{equation*}
Hence by H\"{o}lder's inequality,
\begin{equation*}
	\begin{split}
		&\|\dd f_1\|_{L^p(G)}+\|\db f_1\|_{L^p(G)}
		\\\le& C \pare{\frac{1}{h}\|f\|_{L^p(G)}+\|(2\db\imP)^{-1}\|_{L^{\infty}(G)}\|f\|_{W^{1,p}(G)}+ \|\dd (2\db\imP)^{-1}\|_{L^{p_4^*}(G)}\|f\|_{L^{p_4}(G)}},
	\end{split}
\end{equation*}
where $1/p_4+1/p_4^*\le1/p$ and $0<p<p_4^*<2/\alpha$.
Thus, applying  again Lemma~\ref{lem:dbinvMap} and Sobolev inequalities we obtain {that}
\begin{equation*}
	\begin{split}
		&\|\Sb f\|_{W^{1,p}(G)} 
\\		\le&
		C(\|f_1\|_{W^{1,p}(G)} 
		+\nrm{\idb (e^{2\im\imP/h}\db(\frac{ f}{2\db\imP}))}_{W^{1,p}(G)} )
\\		\le&C\pare{\frac{1}{h}\|f\|_{L^p(G)}+\|(2\db\imP)^{-1}\|_{L^{\infty}(G)}\|f\|_{W^{1,p}(G)}+ \|\dd (2\db\imP)^{-1}\|_{L^{p_4^*}(G)}\|f\|_{L^{p_4}(G)}}
		\\\le &C\pare{
		\frac{1}{h}\|f\|_{L^{p}(G)}
		+\|f\|_{W^{1,p}(G)}
		+ \|f\|_{W^{1,p_5}(G)}}
	\end{split}
\end{equation*}
with $1/p_5\le1/p_4+1/2\le 1/p-1/p_4^*+1/2$ and $1<p<p_4^*<2/\alpha$.
Taking $p_4^*\in[2,2/\alpha)$ then we have that for $h$ sufficiently small,
\begin{equation*}
	\begin{split}
		\|\Sb f\|_{W^{1,p}(G)} 
	\le&C_p	\frac{1}{h}\|f\|_{W^{1,p}(G)} \qquad \mbox{for all $p\in(1,2/\alpha)$}.
	\end{split}
\end{equation*}
Then we further obtain that
\begin{equation*}
	\begin{split}
		&\|\So f\|_{W^{1,p}(G)} 
		=\frac{h}{4}\|\Sb \Sa (qf)\|_{W^{1,p}(G)} 
		\\\le &C\| \Sa (qf)\|_{W^{1,p}(G)}
		\le C
			\|qf\|_{L^{p}(G)}
			\le C
			\|q\|_{L^{\infty}(G)}\|f\|_{L^{p}(G)}
	\end{split}
\end{equation*}
for $h$ sufficiently small and $p\in(1,2/\alpha)$.
\end{proof}
We are now ready to prove Proposition~\ref{prop:CGO1}.
\begin{proof}[Proof of Proposition~\ref{prop:CGO1}]
	We first obtain from Lemma~\ref{lem:Smapping} that, with $1<p<\infty$ and $h$ sufficiently small, {$(I-\So)$} is invertible on $L^p(G)$ and the inverse is bounded. Moreover, the inverse can be written as the Neumann series
	\begin{equation*}
	(I-\So)^{-1}f=\sum_{j=0}^{\infty}\So^j f.
	\end{equation*}
	Define $W$ as in \eqref{eq:CGOform} with $w_h=(I-\So)^{-1}\So A$. We claim that $W$ satisfies $\Delta W+q W=0$ in $G$. In fact, the analyticity of $\Phi$ implies that $\db W=e^{-\Phi/h}\db W_h$ in $G$ and hence
	\begin{equation*}
	\dd\db W=
	\dd\pare{e^{-\Phi/h}\db W_h} 
	=e^{-\overline{\Phi}/h}\dd\pare{e^{ (\overline{\Phi}-\Phi)/h}\db W_h}
	\qquad\mbox{in $G$}.
	\end{equation*} 
	Therefore, $\Delta W+q W=0$ in $G$ if (and only if)  
	\begin{equation*}
	4 e^{2\im\imP/h}\dd\pare{e^{-2\im\imP/h}\db w_h} = -q (A+w_h)\qquad\mbox{where $\imP=\Im\Phi$}.
	\end{equation*}
	\hl{Applying Lemma}~\ref{lem:SoSol}, \hl{the latter can be solved by $w_h=\So (A+w_h)$.} 
	
	We are left with the proof of the estimates. Recall from Lemma~\ref{lem:Smapping} again that
	\begin{equation*}
	\|w_h\|_{L^p(G)}=\|(I-\So)^{-1}\So A\|_{L^p(G)}
	\le C\|\So A\|_{L^p(G)} \le C\text{\hl{$h$}}\| A\|_{L^p(G)},\qquad 1<p<\infty.
	\end{equation*}
	Moreover, by the Neumann series of $(I-\So)^{-1}$ we have $(I-\So)^{-1}\So=\sum_{j=1}^{\infty}\So =\So(I-\So)^{-1}$ and hence
	\begin{equation*}
	\|w_h\|_{W^{1,p}(G)}=\|\So (I-\So)^{-1}A\|_{W^{1,p}(G)}
	\le C\|(I-\So)^{-1} A\|_{L^p(G)} \le C\| A\|_{L^p(G)},\quad 1<p<2/\alpha.
	\end{equation*}
\end{proof}

\section{Local Analysis near a Corner}\label{sec:corner}
In this section, we consider nontrivial solutions $(u,v)$ to \eqref{eq:ITEPlocal} where 
$\gamma$ is scalar in $\weg_{\theta_0,\tilde{\epa}}$. We aim to prove Proposition~\ref{prop:mainLocal}. 

{Let $u,v\in H^1 (\weg_{\theta_0,\tilde{\epa}})$ satisfy} \eqref{eq:ITEPlocal}. Integration by parts yields (see also, \cite{CaX21})
	\begin{equation*}
		\int_{\weg_{\theta_0,\epa}}\pare{\gamma-1} \nabla v\cdot\nabla w  -k^2 (\rho -1) v w\, dx
		=\int_{\partial B_{\epa}\cap \weg_{\theta_0}}\partial_{\nu}^{\gamma}w\pare{v-u}-w\pare{\partial_{\nu}v-\partial_{\nu}^{\gamma}u}\, ds(x)
	\end{equation*}
	{with any $\epa\in(0,\tilde{\epa})$, where $w\in H^1(\weg_{\theta_0,\epa})$ is an arbitrary solution to $\nabla\cdot \gamma \nabla w + k^2\rho  w=0$ in $\weg_{\theta_0,\epa}$.}
In the following, we shall apply Corollary~\ref{cor:CGO} 
{for $w$ with 
	the domain $G$ taken as, for instance, $G=\weg_{\tilde{\theta}_0,\tilde{\epa}}$ for some $\tilde{\theta}_0\in(\theta_0,\pi)$.}
In particular, for any $h\in\RR_+$ which is sufficiently small, there is a solution $w$ of the form \eqref{eq:CGO} with the ``residual'' $w_h$ satisfying the estimates in \eqref{eq:CGOresid}. Notice that for the ``leading term'' $e^{-\Phi/h}$ of $w_h$ we have $\Re\Phi(x)=r^{\alpha}\cos(\alpha\theta)$, where $(r,\theta)$ comes from the polar coordinates of $x$ with $\theta$ taken in $(-\pi,\pi)$. Then we observe for $\alpha\in(0,\frac{1}{2}\pi/\theta_0)$ that
\begin{equation}\label{eq:PhiBound}
	\Re\Phi(x)\ge \delta |x|^{\alpha} \qquad x\in \weg_{\theta_0},
\end{equation}
for some  positive constant $\delta$ independent of $x$.
As a consequence, by similar arguments as in \cite{CaX21} we obtain
\begin{lem}\label{lem:Int2}
	Let $w$ be as specified in Corollary~\ref{cor:CGO} with $\alpha\in(0,\frac{1}{2}\pi/\theta_0)$, $A\equiv 1$ and $h\ll 1$. If $u$ and $v$ satisfy \eqref{eq:ITEPlocal}, then
	\begin{equation}\label{eq:IntK}
		\int_{\weg_{\theta_0,\epa}}\pare{\gamma-1} \nabla v\cdot\nabla w  -k^2 (\rho -1) v w\, dx
		=o(e^{-\epa^{\alpha}\delta/(3h)});
	\end{equation}
and the same is valid for $w$ replaced by its complex conjugate $\overline{w}$.
\end{lem}
In the rest of this section, we shall establish estimates for the left-hand-side integral in \eqref{eq:IntK}. In particular, upper bounds for certain terms in the integral are stated in Section~\ref{sec:upB}, while some sharp estimates are proven in Section~\ref{sec:sharpB}. With these estimates, the proof of Proposition~\ref{prop:mainLocal} is given in Section~\ref{sec:proofProp}. To that end, we need the following preliminary result. 
\begin{lem}\label{lem:Gamma}
	Given $\epsilon,b_0>0$ and $b_1,\mu\in\CC$ with $\Re b_1>0$ and $\Re\mu>0$, we have 
	\begin{equation*}
		\abs{\frac{\Gamma({b_1}/{b_0})}{b_0}\mu^{-b_1/b_0}
			-\int_{0}^{\epsilon}t^{b_1-1}e^{-\mu t^{b_0}}dt
	}
\le \frac{\Gamma(\Re{b_1}/{b_0})}{b_0}(2/\Re\mu)^{b_1/b_0} 
e^{-(\Re{\mu}) \epsilon^{b_0}/2},
	\end{equation*}
where the fractional powers (from the left hand side) are taken as their principal values.
\end{lem}
\begin{proof}
	Recall the integral representation of the Gamma function (see, \cite{DLMF5_9})
	\begin{equation*}
		\int_{0}^{\infty}t^{b_1-1}e^{-\mu t^{b_0}}dt=\frac{\Gamma({b_1}/{b_0})}{b_0}\mu^{-b_1/b_0}. 
	\end{equation*}
Then the left hand side of the desired estimate is an upper incomplete Gamma function, which {satisfies}
 \begin{equation*}
 	\begin{split}
 		\abs{\int_{\epsilon}^{\infty}t^{b_1-1}e^{-\mu t^{b_0}}dt}
 		&\le \int_{\epsilon}^{\infty}t^{\Re{b_1}-1}e^{-(\Re{\mu}) t^{b_0}}dt
 	\le \frac{\Gamma(\Re{b_1}/{b_0})}{b_0}(\Re\mu/2)^{-b_1/b_0} 
 	e^{-(\Re{\mu}) \epsilon^{b_0}/2}.
 	\end{split}
 \end{equation*}
\end{proof}

\subsection{Local Estimates -- Upper Bounds}\label{sec:upB}
We establish in this subsection upper bounds for the integrals at the left-hand-side of \eqref{eq:IntK}.
Recall the function $\Phi$ from \eqref{eq:CGO}. We first have 
\begin{lem}\label{lem:essentialEst}
	For any $\beta\in\RR$ and $p\in(1,\infty]$ with $\beta+2>{2}/{p}$, and any  $f\in L^p(\weg_{\theta_0,\epa})$, we have for $0<h\ll 1$ that
	\begin{equation*}
	\begin{split}
	\abs{\int_{\weg_{\theta_0,\epa}} |x|^{\beta}   e^{-\Phi/h} f(x) dx}
	&=\|f\|_{L^p(\weg_{\theta_0,\epa})}O\pare{h^{\pare{\beta+2-2/p}/{\alpha }}}.
	\end{split}
	\end{equation*}
Moreover, the estimate is also valid for $\overline{\Phi}$.
\end{lem}
\begin{proof}
	Applying \eqref{eq:PhiBound} and Lemma~\ref{lem:Gamma} we obtain that
	\begin{equation*}
	\begin{split}
	\abs{\int_{\weg_{\theta_0,\epa}}
		|x|^{\beta}e^{-\Phi/h} dx}
	\le& {2\theta_0}\int_{0}^{\epa}r^{\beta+1}\esdr  \,dr
	=2\theta_0\frac{\Gamma((\beta+2)/\alpha)}{\alpha\delta^{(\beta+2)/\alpha}}h^{(\beta+2)/\alpha }+O\pare{h^{\beta_0}e^{-\epa^\alpha  \delta /(2h)}}
	\end{split}
	\end{equation*}
	with some finite power $\beta_0$ (independent of $h$), which proves the result for $p=\infty$.
	Similarly for $p\in(1,\infty)$,
	\begin{equation*}
	\begin{split}
	\abs{\int_{\weg_{\theta_0,\epa}} |x|^{\beta}   e^{-\Phi/h} fdx}
	&\le   
	\int_{\weg_{\theta_0,\epa}} \esdr r^{\beta}  |f| \,dx
	\le  \|f\|_{L^p(\weg_{\theta_0,\epa})}
	\,\|\esdr \,r^{\beta}\|_{L^{p^*}(\weg_{\theta_0,\epa})},
	\end{split}
	\end{equation*}
	where $1/p^*+1/p=1$.
	By Lemma~\ref{lem:Gamma} again we have
	\begin{equation*}
	\begin{split}
	\int_{\weg_{\theta_0,\epa}}e^{-p^* \delta\sr /h }\,r^{p^*\beta}\,dx
	&=
	\theta_1\pare{
		\frac{\Gamma\pare{\frac{p^*\beta+2}{\alpha}}}{\alpha\pare{p^*\delta}^{(p^*\beta+2)/\alpha}}h^{(p^*\beta+2)/\alpha}
		+O\pare{h^{\beta_0}e^{-\epa^\alpha p^* \delta /(2h)}}
	}
	\end{split}
	\end{equation*}
	with $\beta p^*+2>0$ and $\beta_0$ a finite constant independent of $h$. Hence
	\begin{equation*}
	\|\esdr \,r^{\beta}\|_{L^{p^*}(\weg_{\theta_0,\epa})}
	\le Ch^{(\beta+2/p^*)/\alpha}=Ch^{(\beta+2-2/p)/\alpha}.
	\end{equation*}
The estimates for $\overline{\Phi}$ are verified in exactly the same way. 
\end{proof}
Applying Lemma~\ref{lem:essentialEst} we obtain
\begin{lem}\label{lem:localAll}
	Let $v$ be a real analytic function on $B_{\epa}$. 
	Denote $v_{N_0}$ as the first nonzero (and the $N_0$-th order) term from the Taylor series of $v$ at $x_0=0$, and $V_{N}$ as the first nonzero (and the $N$-th order) term from that of $\nabla v$.
	Let $w$ be a function of the form \eqref{eq:CGO} with $h\ll 1$, $\alpha\in(0,\frac{1}{2}\pi/\theta_0)$, $A\equiv 1$, $\|w_h\|_{L^{p_1}(\weg_{\theta_0,\epa})}\le Ch$ and $\nabla w_h\in L^{p_2}(\weg_{\theta_0,\epa})$ for some $p_1>2/\alpha$, $p_2>2$ and $C\ge 0$.
	Suppose that there are constants $c_j$, $C_j$ and $\beta_j\ge 0$, $j=1,2$, such that 
	\begin{equation*}
	\abs{\pare{\gamma^{-1/2}\pare{\gamma-1}-c_1}}\le C_1 r^{\beta_1}
	\AnD
	\abs{\pare{\gamma^{-1/2}\pare{\rho-1}-c_2}}\le C_2 r^{\beta_2}
	\qquad\mbox{a.e. in $\weg_{\theta_0,\epa}$}.
	\end{equation*}
	Then 
	\begin{equation*}
	\begin{split}
	&\abs{\int_{\weg_{\theta_0,\epa}}\pare{\gamma-1} \nabla v\cdot\nabla w \,dx}=
	\pare{\jump_1+h^{\beta_1/{\alpha }}
}O(h^{\pare{N+1}/{\alpha }}),
	\\&
	\abs{\int_{\weg_{\theta_0,\epa}}\pare{\gamma-1} \nabla v\cdot\nabla w \,dx
	+\frac{\jump_1}{h}\int_{\weg_{\theta_0,\epa}}
	V_N(x)\cdot (\nabla\Phi) e^{-\Phi/h} dx}
	\\&=
\pare{h^{\beta_1/{\alpha }}+\jump_1h^{1-2/(\alpha p_1)}
	+\jump_1h^{\pare{1-2/{p_2}}/{\alpha }}}
O(h^{\pare{N+1}/{\alpha }})	,
\\&
	\abs{\int_{\weg_{\theta_0,\epa}} (\rho-1) v w \,dx}
=
\pare{\jump_2+h^{\beta_2/{\alpha }}}O(h^{\pare{N_0+2}/{\alpha }}),
\\
	&\abs{\int_{\weg_{\theta_0,\epa}} (\rho-1) v w \,dx
	-c_2\int_{\weg_{\theta_0,\epa}}
	v_{N_0}(x) e^{-\Phi/h} dx}
	=\pare{h^{\beta_2/{\alpha }}+\jump_2h^{1-2/(\alpha p_1)}}O(h^{\pare{N_0+2}/{\alpha }}).
\end{split}
	\end{equation*}
Moreover, all the estimates also hold for $\overline{w}$, in which case $\Phi$ is replaced by $\overline{\Phi}$.
\end{lem}
\begin{proof}
	Throughout the proof, the notation $\lesssim$ between two quantities depending on $h$, say, $f(h)$ and $g(h)$, means $f(h)\le C g(h)$ for $h$ sufficiently small, where $C$ is a positive constant independent of $h$.
	
	Denote $\vec{b}=\vec{b}_{\gamma}=-\gamma^{1/2}\nabla\gamma^{-1/2}$.
	Then for $w$ of the form \eqref{eq:CGO} we have
	\begin{equation*}
		(\gamma-1)	\nabla w = (\gamma-1)\gamma^{-1/2}e^{-\Phi/h}\pare{-(A+w_h) (\frac{1}{h}\nabla\Phi+\vec{b})+ \nabla A+\nabla w_h}
	\end{equation*}
	with 
	\begin{equation}\label{eq:gradPhi}
		\nabla \Phi
		=\alpha r^{\alpha-1}\,e^{\im (\alpha-1)\theta}\begin{bmatrix}
			1\\\im
		\end{bmatrix}.
	\end{equation}
	Since $A\equiv 1$, we can split the integral 
	\begin{equation*}
	\int_{\weg_{\theta_0,\epa}}
	\pare{\gamma-1} \nabla v({x}) \cdot\nabla w\,dx
	=-c_1I_{10}+\sum_{j=1}^{3}I_{1j},
	\end{equation*}
	where the integrals $I_{1j}$, $j=0,\ldots,3$, depending on $h$, are defined by
	\begin{equation*}
	\begin{split}
	I_{10}&:=\frac{1}{h}\int_{\weg_{\theta_0,\epa}}
	e^{-\Phi/h}\nabla\Phi\cdot V_N   dx,
	\\
	I_{11}&:=\int_{\weg_{\theta_0,\epa}}
	\pare{\gamma^{-1/2}\pare{\gamma-1}-c_1}\nabla v\cdot \pare{-(1+w_h)(\frac{1}{h}\nabla\Phi+\vec{b})+ \nabla w_h}\, e^{-\Phi/h} dx,
	\\
	I_{12}&: =\jump_1\int_{\weg_{\theta_0,\epa}}
	\pare{\nabla v -V_N}\cdot \pare{-(1+w_h)(\frac{1}{h}\nabla\Phi+\vec{b})+ \nabla w_h}\, e^{-\Phi/h} dx,
	\\
	I_{13}&: =\jump_1\int_{\weg_{\theta_0,\epa}}
	V_N\cdot \pare{-w_h(\frac{1}{h}\nabla\Phi+\vec{b})-\vec{b}+ \nabla w_h}\, e^{-\Phi/h} dx.
	\end{split}
	\end{equation*}		
	Notice in $\weg_{\theta_0,\epa}$  (in the $L^{\infty}$ sense) that 
	\begin{equation*}
	\abs{V_N}\lesssim r^N,
	\qquad \abs{\nabla\Phi}\lesssim r^{\alpha-1},
	\qquad \abs{\pare{\gamma^{-1/2}\pare{\gamma-1}-c_1}}\lesssim r^{\beta_1},
	\qquad
	\abs{\nabla v- V_N}\lesssim r^{N+1 }.
	\end{equation*}
	Applying Lemma~\ref{lem:essentialEst} with $\beta=N+\alpha-1$ and $p=\infty$ we obtain that
	\begin{equation*}
	\abs{I_{10}}\lesssim h^{\pare{N+\alpha+1}/{\alpha }-1}\lesssim h^{\pare{N+1}/{\alpha }},
	\end{equation*}
and similarly (with $\beta=N$),
\begin{equation*}
	\abs{\int_{\weg_{\theta_0,\epa}}
		V_N\cdot \vec{b} e^{-\Phi/h} dx}
	\lesssim h^{(N+2)/\alpha}.
\end{equation*}
	By Lemma~\ref{lem:essentialEst} again but with $f=w_h$ or $f=\nabla w_h$ we derive that
	\begin{equation*}
		\abs{\int_{\weg_{\theta_0,\epa}}
		 	V_N\cdot(\frac{1}{h}\nabla\Phi+\vec{b}) e^{-\Phi/h} w_hdx}
	 	\lesssim \|w_h\|_{L^{p_1}(\weg_{\theta_0,\epa})}h^{\pare{N+1-2/p_1}/{\alpha }}
	\end{equation*} 
and
\begin{equation*}
	\abs{\int_{\weg_{\theta_0,\epa}}
	e^{-\Phi/h}	V_N\cdot \nabla w_h  dx}
	\lesssim \|\nabla w_h\|_{L^{p_2}(\weg_{\theta_0,\epa})}h^{\pare{N+2-2/p_2}/{\alpha }}.
\end{equation*}
Therefore, we have deduced that
	\begin{equation*}
	\begin{split}
	\abs{I_{13}}	
&\lesssim \jump_1
\pare{h^{\pare{N+2}/{\alpha }}+
\|w_h\|_{L^{p_1}(\weg_{\theta_0,\epa})}h^{\pare{N+1-2/{p_1}}/{\alpha }}
+	\|\nabla w_h\|_{L^{p_2}(\weg_{\theta_0,\epa})}h^{\pare{N+2-2/{p_2}}/{\alpha }}}
\\&\lesssim \jump_1 h^{(N+1)/\alpha}
(h^{1-2/(\alpha p_1)}
+h^{\pare{1-2/{p_2}}/{\alpha }}).	
\end{split}
	\end{equation*}
	Similarly,
	\begin{equation*}
	\begin{split}
	\abs{I_{12}}
&\lesssim \jump_1
	\pare{h^{\pare{N+2}/{\alpha }}+
		\|w_h\|_{L^{p_1}(\weg_{\theta_0,\epa})}h^{\pare{N+2-2/{p_1}}/{\alpha }}
		+	\|\nabla w_h\|_{L^{p_2}(\weg_{\theta_0,\epa})}h^{\pare{N+3-2/{p_2}}/{\alpha }}
}
\\&\lesssim \jump_1 h^{(N+2)/\alpha},	
	\end{split}
	\end{equation*}
	and
	\begin{equation*}
	\begin{split}
	\abs{I_{11}}
&\lesssim h^{\pare{N+1+\beta_1}/{\alpha }}+
		\|w_h\|_{L^{p_1}(\weg_{\theta_0,\epa})}h^{\pare{N+1+\beta_1-2/{p_1}}/{\alpha }}
		+	\|\nabla w_h\|_{L^{p_2}(\weg_{\theta_0,\epa})}h^{\pare{N+2+\beta_1-2/{p_2}}/{\alpha }}
		\\&\lesssim h^{\pare{N+1+\beta_1}/{\alpha }}.
	\end{split}
	\end{equation*}

In the same manner we write
\begin{equation*}
\int_{\weg_{\theta_0,\epa}}
(\rho-1) v w\,dx
 =\jump_2I_{20}+\sum_{j=1}^{4}I_{2j},
\end{equation*}
where $I_{2j}$, $j=0,...,3$, depending on $h$, are defined {by}
\begin{equation*}
\begin{split}
I_{20}&:=\int_{\weg_{\theta_0,\epa}}
v_{N_0} e^{-\Phi/h} dx,
\qquad
I_{21}:=\int_{\weg_{\theta_0,\epa}}
\pare{\gamma^{-1/2}\pare{\rho-1}-c_2}  v\,(1+ w_h)\, e^{-\Phi/h} dx,
\\
I_{22}&: =\jump_2\int_{\weg_{\theta_0,\epa}}
\pare{ v -v_{N_0}}(1+ w_h)\, e^{-\Phi/h} dx,
\qquad
I_{23}: =\jump_2\int_{\weg_{\theta_0,\epa}}
v_{N_0}w_h\, e^{-\Phi/h} dx.
\end{split}
\end{equation*}
In $\weg_{\theta_0,\epa}$ we have (in the $L^{\infty}$ sense)
\begin{equation*}
\abs{v_{N_0}}\lesssim r^{N_0}\qquad  \abs{\pare{\gamma^{-1/2}\pare{\rho-1}-c_1}}\lesssim r^{\beta_2}
\qquad
\abs{v- v_{N_0}}\lesssim r^{N_0+1 }.
\end{equation*}
Applying Lemma~\ref{lem:essentialEst} twice with $\beta=N_0$ and, respectively, $f\equiv 1$ and $f=w_h$, we obtain that
\begin{equation*}
\abs{I_{20}}
\lesssim h^{\pare{N_0+2}/{\alpha }}
\AND
\abs{I_{23}}\lesssim 
 \jump_2 \|w_h\|_{L^{p_1}(\weg_{\theta_0,\epa})} h^{\pare{N_0+2-2/p_1}/{\alpha }}
\lesssim 
\jump_2h^{\pare{N_0+2+\alpha-2/p_1}/{\alpha }}.
\end{equation*}
Similarly, we have
\begin{equation*}
\abs{I_{22}}
\lesssim \jump_2 \pare{h^{\pare{N_0+3}/{\alpha }}+ \|w_h\|_{L^{p_1}(\weg_{\theta_0,\epa})} h^{\pare{N_0+3-2/p_1}/{\alpha }}}
\lesssim \jump_2 h^{\pare{N_0+3}/{\alpha }}
\end{equation*}
and
\begin{equation*}
\abs{I_{21}}\lesssim 
h^{\pare{N_0+2+\beta_2}/{\alpha }}+ 
\|w_h\|_{L^{p_1}(\weg_{\theta_0,\epa})} h^{\pare{N_0+2+\beta_2-2/p_1}/{\alpha }}
\lesssim 
h^{\pare{N_0+2+\beta_2}/{\alpha }}.
\end{equation*}

The estimates for $\overline{w}$ can be proven in the same way. 
\end{proof}

\subsection{Sharp Local Estimates}\label{sec:sharpB}
We provide sharp estimates for the integrals at the left-hand-side of \eqref{eq:IntK} for $\weg_{\theta_0,\epa}$ of aperture $2\theta_0$ with $\theta_0\in(0,\pi)\setminus\{\pi/2\}$. To that end, we need some properties of solutions to the Helmholtz {equation.}

{The following} lemma states some properties of the first two non-zero terms in the Taylor series of a solution to the Helmholtz equation. The proof is based on straightforward comparison of degrees of homogeneous polynomials in each term of the Taylor series. We choose to skip the proof here and refer the readers to \cite{BPS14,CaX21} for the arguments.  
\begin{lem}\label{lem:TaylorvV}
	Let $v$ be a solution to the Helmholtz equation {$\Delta v +k^2 v=0$} in a neighborhood of the origin.	
	For each $j\in\mathbb{N}$, denote $v_{j}(x)$ as the $j$-th order term in the Taylor series of $v$, which is a homogeneous polynomials of degree $j$, and $V_j(x)$ be that of $\nabla V$. 
	Suppose that $v_{N_0}$ is the first non-zero term in the Taylor series of $v$, and $V_{N}$ that of $\nabla v$. 
	Then 
	\begin{enumerate}
		\item The terms $v_{N_0}$, $v_{N_0+1}$, $V_{N}$ and $V_{N+1}$ are harmonic.
		\item If $v$ vanishes at the origin, that is, when $N_0\ge 1$, then $N=N_0-1$. 
		\item If $v$ is nonzero at the origin, namely, $N_0=0$, then either $N=0$ and $V_0=\nabla v_1$, or $N=1$, $v_1=0$, $V_1=\nabla v_2$ and $\nabla\cdot V_1=-k^2v_0$. 	
	\end{enumerate}
\end{lem}

Recall $\Phi$ given by \eqref{eq:CGO}. In the rest of this  section, denote 
\begin{equation*}
	\Phi_1=\Phi\AND \Phi_2=\overline{\Phi}.
\end{equation*}
We have the following sharp estimates.
\begin{lem}\label{lem:DecayExactv}
	Let $v_{N_0}$ be a (nontrivial) homogeneous polynomial of degree $N_0\in\mathbb{N}$ which is harmonic. Then 
	\begin{equation}\label{eq:decayIntv}
		\int_{\weg_{\theta_0,\epa}}v_{N_0}(x)e^{-\Phi_j/h} \,dx=C_{0,N_0,j}\,h^{(N_0+2)/\alpha}+o\pare{e^{-\epa^\alpha \delta /(3h)}},
		\qquad j=1,2,
	\end{equation}
	where $C_{0,N_0,j}$, $j=1,2$, are constants independent of $h$, and at least one of them is nonzero. 
\end{lem}
\begin{proof}
	Applying  Lemma~\ref{lem:Gamma} we {have}
	\begin{equation*}
		\begin{split}
			&	\int_{\weg_{\theta_0,\epa}}r^{N_0}e^{\pm\im  N_0\theta}e^{-\Phi_j/h}  \,dx	
			=\int_{-\theta_0}^{\theta_0}e^{\pm\im  N_0\theta}
			\int_{0}^{\epa}r^{N_0+1}e^{-r^\alpha (\cos\alpha\theta-(-1)^j\im\sin\alpha\theta)/h} drd\theta
			\\&
			=\alpha^{-1}\Gamma((N_0+2)/\alpha){h}^{(N_0+2)/\alpha}
			\int_{-\theta_0}^{\theta_0} 
			e^{\pm\im  N_0\theta}e^{(-1)^j\im \alpha\theta (N_0+2)/\alpha}   d\theta
			+O\pare{h^{\beta_0} e^{-\epa^\alpha \delta /(2h)}}
			\\&=
			\frac{	-\im\Gamma((N_0+2)/\alpha){h}^{(N_0+2)/\alpha}}{\pare{((-1)^j\pm 1)N_0+(-1)^j 2}\alpha}
			\pare{ e^{\im  (((-1)^j\pm 1)N_0+(-1)^j 2)\theta_0} -e^{-\im  (((-1)^j\pm 1)N_0+(-1)^j 2)\theta_0}}
			\\&\qquad	+o\pare{  e^{-\epa^\alpha \delta /(3h)}},
		\end{split}	  
	\end{equation*}
	where $\beta_0$ is some constant independent of $h$.
	Hence for any constants $B_1$ and {$B_2$,}
	\begin{equation*}
		\begin{split}
			&	\frac{\im\alpha}{\Gamma((N_0+2)/\alpha)\,h^{(N_0+2)/\alpha}}
			\int_{\weg_{\theta_0,\epa}}r^{N_0}
			\pare{B_1e^{\im  N_0\theta}+B_2e^{-\im  N_0\theta}}e^{-\Phi_j/h}  \,dx	
			\\&=
			\frac{(-1)^j}{ 2}	
			\pare{{B_j}
				( e^{(-1)^j\im   2\theta_0} -e^{-(-1)^j\im   2\theta_0})
				+B_{j+1}\frac{e^{(-1)^j\im 2 (N_0+1 )\theta_0}-e^{-(-1)^j\im 2 (N_0+1 )\theta_0}}{N_0+1}}
		\\&\qquad	+o\pare{  e^{-\epa^\alpha \delta /(3h)}},
		\end{split}	  
	\end{equation*}	
	where $B_{2+1}:=B_1$.
	Thus, \eqref{eq:decayIntv} is verified by taking $v_{N_0}=B_1e^{\im  N_0\theta}+B_2e^{-\im  N_0\theta}$.
	It is left to show the nontriviality of the constants $C_{0,N_0,j}$, $j=1,2$. Recall that $B_1^2+B_2^2\neq 0$. If $C_{0,N_0,1}=C_{0,N_0,2}=0$, {then }
	\begin{equation*}
		\begin{split}
			0&=\det\pare{\begin{matrix}
					e^{-2\im\theta_0}-e^{2\im\theta_0}
					&\dfrac{e^{-2\im(N_0+1)\theta_0}-e^{2\im(N_0+1)\theta_0}}{N_0+1}
					\\-\dfrac{e^{-2\im(N_0+1)\theta_0}-e^{2\im(N_0+1)\theta_0}}{N_0+1}
					& -(e^{-2\im\theta_0}-e^{2\im\theta_0})
			\end{matrix}}
.
		\end{split}
	\end{equation*}
	That is,
	\begin{equation*}
		\sin\theta_0=\pm\frac{\sin(N_0+1)\theta_0}{N_0+1},
	\end{equation*}
	which is impossible since $\theta_0\in(0,\pi)\setminus\{\pi/2\}$. 
\end{proof}
\begin{rem}
	We observe from the proof that, in the case when $v_{N_0}\equiv 1$ ($N_0=0$), the constants $C_{0,N_0,j}$ are given by
	\begin{equation*}
		C_{0,0,j}=\frac{\Gamma(2/\alpha)}{(-1)^j 2\im\alpha}
		( e^{(-1)^j\im   2\theta_0} -e^{ -(-1)^j \im2\theta_0}),\qquad j=1,2.
	\end{equation*}
\end{rem}
\begin{lem}\label{lem:DecayExact}
	Given $N\in \mathbb{N}$, let $V=V_N(x)$ be the gradient of a homogeneous polynomial of degree $N+1$ which is harmonic. 
	Suppose that $V$ is not identically zero.
	Then we have 
	\begin{equation}\label{eq:decayInt}
	-\frac{1}{h}\int_{\weg_{\theta_0,\epa}}e^{-\Phi_j/h} V\cdot \nabla\Phi_j \,dx=C_{1,j}\,{{h}^{(N+1)/\alpha}+o\pare{e^{-\epa^{\alpha} \delta /(3h)}}},\qquad j=1,2
	\end{equation}
	with constants $\beta_0$ and $C_{1,j}$ independent of ${h}$, $j=1,2$.
	Moreover, $C_{1,1}=C_{1,2}=0$,  if and only if the aperture of the corner
	\begin{equation}\label{eq:Angle}
	2\theta_0=\frac{l\pi}{1+N}\in(0,2\pi)\backslash\{\pi\},\qquad   \mbox{i.e., } N=\frac{\pi}{2\theta_0}l-1\in\mathbb{N}_+,
	\end{equation}
	for some $l\in\mathbb{N}_+$.
\end{lem}
\begin{rem}\label{rem:N=0}
	We observe from \eqref{eq:Angle} that, if $N=0$ then $C_{1,1}$ and $C_{1,2}$ can never be both zero.
	In the case if $N=1$, the only possibility for $C_{1,1}=C_{1,2}=0$ is when $\theta_1=\pi/2$ or $3\pi/2$.
\end{rem}
\begin{rem}\label{rem:robustA}
	We note that, the situation $C_{1,1}=C_{1,2}=0$, that is, when $2\theta_0$ and $N$ satisfies the relation \eqref{eq:Angle}, is ``invariant'' under different (admissible) choices of the holomorphic function $A$, as in \eqref{eq:CGO} for $w$. More precisely, when \eqref{eq:Angle} is satisfied, by the same arguments as in the proof of Lemma~\ref{lem:DecayExact} we have for $A$ with $\db A=0$ that 
	\begin{equation*}
	\int_{\weg_{\theta_0,\epa}} V\cdot \nabla (A_j e^{-\Phi_j/h})  \,dx=o\pare{h^{(N+1)/\alpha}}, \qquad j=1,2,
	\end{equation*}
with $A_1=A$ and $A_2=\overline{A}$, provided that $A(x)=O(1)$ near the corner.
If $A(x)=O(|x|^{\beta})$ around the corner, then the left hand side of the above relation is replaced by $o(h^{(N+1+\beta)/\alpha})$. 
We 
omit the proof of this fact to keep our main context concise. 
\end{rem}
\begin{proof}[Proof of Lemma~\ref{lem:DecayExact}]
	Recall that $r^{N}e^{\pm\im N\theta}$ form a basis of all homogeneous harmonic polynomial of degree $N$, where $x_1$ and $x_2$ denote the Cartesian components of $x\in\RR^2$. Then we can write the vector field $V$ as
	\begin{equation*}
	\begin{split}
	V(x)&=b_1 \begin{bmatrix}
		1\\\im
		\end{bmatrix}r^N e^{\im N\theta}
	+b_2\begin{bmatrix}
		1\\-\im
		\end{bmatrix}r^N e^{-\im N\theta}.
	\end{split}
	\end{equation*}
	Recalling \eqref{eq:gradPhi} we have 
	\begin{equation*}
		\begin{split}
			V\cdot \nabla\Phi_j
			=&2\alpha b_{j+1}\,r^{N-1+\alpha}  e^{(-1)^{j}\im (N+1-\alpha)\theta}
			\qquad j=1,2,
		\end{split}
	\end{equation*}
where $b_{2+1}:=b_1$.
	Applying Lemma~\ref{lem:Gamma} {yields}
	\begin{equation*}
	\begin{split}
	&\frac{1}{2b_{j+1}\alpha h}\int_{\weg_{\theta_0,\epa}}e^{-\Phi_j/h} V\cdot \nabla\Phi_j dx
	\\=&
	\Gamma(\frac{N+1}{\alpha}+1)
	\,h^{(N+1)/\alpha}
	\int_{-\theta_0}^{\theta_0}e^{(-1)^{j}\im 2(N+1) \theta}d \theta 
	+ o\pare{e^{-\epa^\alpha  \delta /(3h)}}
	\end{split}	  
	\end{equation*}
with some finite power $\beta_0$ (independent of $h$).
	Then \eqref{eq:decayInt} is verified by the simple fact that
	\begin{equation*}
	\begin{split}
	\int_{-\theta_0}^{\theta_0}e^{(-1)^{j}\im 2(N+1) \theta}d \theta 
	=	 -\im\frac{(-1)^{j}}{2 (N+1)}	\pare{e^{(-1)^{j}\im 2(N+1) \theta_0}-e^{-(-1)^{j}\im 2(N+1) \theta_0}},
	\end{split}
	\end{equation*}
with the constants
\begin{equation*}
	C_{1,j}=\im (-1)^{j} \frac{b_{j+1}\alpha}{N+1}	\pare{e^{(-1)^{j}\im 2(N+1) \theta_0}-e^{-(-1)^{j}\im 2(N+1) \theta_0}},\qquad j=1,2.
\end{equation*}
Finally, for either $j=1$ or $2$, $C_{1,j}=0$ if and only if $b_{j+1}=0$ or $e^{(-1)^{j}\im 4(N+1) \theta_0}=1$, where the latter is equivalent to that $4(N+1)\theta_0=2l\pi$ for some integer $l$.
The proof is complete. 
\end{proof}

\begin{lem}\label{lem:DecayExact3}
	Let $V(x)$ be the gradient of a homogeneous polynomial of degree $2$ and satisfy $\nabla\cdot V=-k^2v_0\ne 0$. 
	Given two constants $\jump_1$ and $\jump_2$ with $\jump_1\neq 0$, 
	then for each $j=1,2$ we have
	\begin{equation}\label{eq:decayInt3}
	-	\frac{\jump_1}{h}\int_{\weg_{\theta_0,\epa}} e^{-\Phi_j/h} \,V\cdot \nabla\Phi_j \,dx 
	-k^2v_0\jump_2\int_{\weg_{\theta_0,\epa}} e^{-\Phi_j/h}  \,dx=C_{2,j} h^{2/\alpha}+o\pare{  e^{-\epa^\alpha \delta /(3h)}}, 
	\end{equation}
	with some constant 
	$C_{2,j}$ independent of $h$. 	
	Moreover, 
	$C_{2,1}=C_{2,2}=0$ happens only when either $2\theta_0\in\{\pi/2,3\pi/2\}$ and $\jump_1=\jump_2$, or $2\theta_0\notin\{\pi/2,3\pi/2\}$ and
	\begin{equation*}
	V(x)=-\frac{1}{4}k^2v_0\nabla ((1-c_0)x_1^2+(1+c_0)x_2^2)
	\qquad\mbox{with}\quad c_0=\frac{1-\jump_2/\jump_1}{\cos2\theta_0}.
	\end{equation*}
\end{lem}
\begin{proof}
We can write
	$V(x)=\frac{1}{2}\nabla(x^TBx)=Bx$ where $B=(b_{jk})$ is a $2$-by-$2$ symmetric matrix with $b_{11}+b_{22}=-k^2v_0$.
	Recall from \eqref{eq:gradPhi} {that}
	\begin{equation*}
		\begin{split}
			V\cdot \nabla \Phi_j
			&
			=\alpha r^{\alpha-1}\,e^{-(-1)^j\im (\alpha-1)\theta}
			\begin{bmatrix}
				1&-(-1)^j\im
			\end{bmatrix}
		Bx
			\\	&=\alpha r^{\alpha}	(b_j e^{(-1)^j\im(2-\alpha)\theta}-k^2 v_0 e^{-(-1)^j\im\alpha\theta})/2
			,\qquad j=1,2,
		\end{split}
	\end{equation*}
where $b_j=b_{11}-b_{22}-(-1)^j2\im b_{12}$, $j=1,2$.
{Then}
\begin{equation*}
	\begin{split}
		&\frac{2}{ h}\int_{\weg_{\theta_0,\epa}} e^{-\Phi_j/h} \,V\cdot \nabla\Phi_j \,dx 
		\\=	&	
		\frac{\alpha}{ h}\int_{-\theta_0}^{\theta_0}(b_je^{(-1)^j\im(2-\alpha)\theta}-k^2 v_0 e^{-(-1)^j\im\alpha\theta})
		\int_{0}^{\epa}r^{1+\alpha}e^{-\Phi_j/h}drd\theta
		\\=&
		{\Gamma(\frac{2}{\alpha}+1)}h^{2/\alpha}
		\int_{-\theta_0}^{\theta_0}
		b_je^{(-1)^j\im 4\theta}-k^2 v_0 e^{(-1)^j\im 2\theta}
		d\theta
		+o\pare{  e^{-\epa^\alpha \delta /(3h)}}.
	\end{split}
\end{equation*}
Therefore, we obtain that
\begin{equation*}
	\frac{1}{h}\int_{\weg_{\theta_0,\epa}}e^{-\Phi_j/h} \,V\cdot \nabla\Phi \,dx
	=\widetilde{C}_{1,j} h^{2/\alpha}+o\pare{  e^{-\epa^\alpha \delta /(3h)}}
\end{equation*}
{with} 
\begin{equation*}
	\begin{split}
		\widetilde{C}_{1,j}
		&=\frac{\Gamma({2}/{\alpha})}{(-1)^j\im 2\alpha}
		(e^{(-1)^j\im 2\theta_0}-e^{-(-1)^j\im 2\theta_0})
		\pare{\frac{b_j}{2}(e^{(-1)^j\im 2\theta_0}+e^{-(-1)^j\im 2\theta_0})
			-k^2v_0}.
	\end{split}
\end{equation*}
Moreover, recalling \eqref{eq:decayIntv} 
we achieve \eqref{eq:decayInt3} with the {constants}
\begin{equation*}
	\begin{split}
		C_{2,j}&
		=\frac{\im\jump_1\Gamma({2}/{\alpha})}{(-1)^j 2\alpha}
		( e^{(-1)^j\im   2\theta_0} -e^{- (-1)^j 2\theta_0})
		\pare{\frac{b_j}{2}(e^{(-1)^j\im 2\theta_0}+e^{-(-1)^j\im 2\theta_0})
			-k^2v_0 (1-\frac{\jump_2}{\jump_1})}.
	\end{split}
\end{equation*}
	Notice that $e^{\mp\im 4\theta_0}\neq 1$ since $2\theta_0\in(0,2\pi)\setminus\{\pi\}$, and that $e^{\mp\im 4\theta_0}= -1$ if and only if $2\theta_0=\pi/2, 3\pi/2$. 
If $2\theta_0\in\{\pi/2,3\pi/2\}$ then ${C}_{2,j}=0$ only when $\jump_1=\jump_2\neq 0$, $j=1,2$. Otherwise if $2\theta_0\notin\{\pi/2,3\pi/2\}$, then $C_{2,1}=C_{2,2}=0$ if and only if 
\begin{equation*}
	b_{j}
	=2k^2v_0\pare{1- \frac{\jump_2}{\jump_1}} \frac{1}{e^{(-1)^j\im   2\theta_0} +e^{-(-1)^j\im 2\theta_0}}
	=k^2v_0\pare{1- \frac{\jump_2}{\jump_1}} \frac{1}{\cos 2\theta_0},
\end{equation*}
which is equivalent to that $b_{11}=-k^2v_0(1-c_0)/2$, $b_{22}=-k^2v_0(1+c_0)/2$ and $b_{12}=b_{21}=0$.
The proof is complete.
\end{proof}

\subsection{Proof of Proposition~\ref{prop:mainLocal}}\label{sec:proofProp}
We are finally in a position to prove Proposition~\ref{prop:mainLocal}.

Let $\rho$, $\gamma$ be the same as those in Theorem~\ref{thm:main} with a corner of aperture $2\theta_0$ where $(\gamma,\rho)$ is admissible. Denote $\jump_1$ and $\jump_2$ as the essential jumps of, respectively, $\gamma$ and $\rho$ at the corner as defined in Definition~\ref{def:corner}. We apply the rigid change of coordinates as in Definition~\ref{def:corner}; in particular, the corner is located at the origin under the new coordinates.

Given $(u,v)$ a nontrivial solution to \eqref{eq:ITEPlocal}, then Lemma~\ref{lem:Int2} states 
\begin{equation}\label{eq:IntDecay}
	\int_{\weg_{\theta_0,\epa}}\pare{\gamma-1} \nabla v\cdot\nabla w_j  -k^2 (\rho -1) v w_j\, dx
	=o(e^{-\epa^{\alpha}\delta/(3h)}),\qquad j=1,2,
\end{equation}
where $h$ is an arbitrary positive constant which is sufficiently small, $w_1=w$, $w_2=\overline{w}$ and $w$ is the solution to $\nabla\cdot\gamma\nabla w+k^2\rho w=0$ as specified in Corollary~\ref{cor:CGO}. In particular, $w$ is of the form \eqref{eq:CGO} with $\alpha\in(0,\frac{1}{2}\pi/\theta_0)$, $A\equiv 1$ and the estimates \eqref{eq:CGOresid} for any $p_1\in {(2/\alpha,\infty)}$ and $p_2\in(2,2/\alpha)$. 

If $v$ can be extended into $B_{\tilde{\epa}}$ while remaining as a solution to $\Delta v+k^2v=0$, then $v$ is real-analytic in $B_{\epa}$, in which case we can consider its Taylor series. 
Denote $N_0$ as the vanishing order of $v$ at the corner and $N$ as that of $\nabla v$. In other words, the first non-zero term of the Taylor series of $v$ is the $N_0$-th order term, which is a (harmonic) homogeneous polynomial of degree $N_0$ in $x$ and is denoted as $v_{N_0}(x)$. Similarly, the first non-zero term in the Taylor series of $\nabla v$, denoted as $V_{N}(x)$, is a harmonic and homogeneous polynomial (vector) of degree $N$.

We first show that such an extension for $v$ never happens in the case of Item (I-a) from Theorem~\ref{thm:main}. That is to say, if for some constants $\beta_1>2$, $\beta_2>0$ and $c_2\neq 0$,
\begin{equation*}
	\abs{\gamma^{-1/2}\pare{\gamma-1}}\le C_1 |x|^{\beta_1}
	\AnD
	\abs{\gamma^{-1/2}\pare{\rho-1}-c_2}\le C_2 |x|^{\beta_2}
	\quad\mbox{a.e. in $\weg_{\theta_0,{\epa}}$},
\end{equation*} 
then $v$ cannot be extended into $B_{\tilde{\epa}}$ as a solution to $\Delta v+k^2v=0$. 
In fact, assuming that $v$ admits such an extension, then Lemma~\ref{lem:localAll} implies that
\begin{equation*}
\abs{\int_{\weg_{\theta_0,\epa}}\pare{\gamma-1} \nabla v\cdot\nabla w_j -k^2(\rho-1) v w_j+ k^2c_2
			v_{N_0} e^{-\Phi_j/h}\,dx}
		\le C h^{\tilde{\beta}_1/\alpha},
\qquad j=1,2,
\end{equation*}
with $\tilde{\beta}_1=\min\{N+1+\beta_1,\,N_0+2+\beta_2,\,N_0+2+\alpha-2/p_1\}$.
Recall from Lemma~\ref{lem:TaylorvV} that $N\ge N_0-1$ and hence $\tilde{\beta}_1>N_0+2$.
On the other hand we have from Lemma~\ref{lem:DecayExactv} that 
\begin{equation*}
k^2c_2\int_{\weg_{\theta_0,\epa}}
		v_{N_0} e^{-\Phi_j/h}\,dx
	= k^2c_2C_{0,N_0,j} h^{(N_0+2)/\alpha}+o(e^{-\epa^{\alpha}\delta/(3h)}),
	\qquad j=1,2,
\end{equation*}
where $C_{0,N_0,1}$ and $C_{0,N_0,2}$ are constants independent of $h$; moreover, at least one of these two constants is nonzero.
Then we arrive at a contradiction against \eqref{eq:IntDecay}. 

Next, we apply similar arguments for the case of 
\begin{equation*}
	\abs{\gamma^{-1/2}\pare{\gamma-1}-c_1}\le C_1 |x|^{\beta_1}
	\AnD
	\abs{\gamma^{-1/2}\pare{\rho-1}-c_2}\le C_2 |x|^{\beta_2}
	\quad\mbox{a.e. in $\weg_{\theta_0,{\epa}}$}
\end{equation*} 
with some constants $c_j\in\RR$ and $\beta_j>0$, $j=1,2$, with $c_1\neq 0$. 
By Lemma~\ref{lem:localAll} again we obtain that
\begin{equation*}
	\begin{split}
&
		\abs{\int_{\weg_{\theta_0,\epa}} (\gamma-1) \nabla v\cdot\nabla w_j-k^2(\rho-1) v w_j
			+\frac{\jump_1}{h}
			V_N(x)\cdot (\nabla\Phi_j) e^{-\Phi_j/h} 
			+ k^2c_2
			v_{N_0} e^{-\Phi_j/h}dx}
\\&		\le Ch^{\tilde{\beta}_2/\alpha},\qquad j =1,2,\qquad\text{with}\quad \tilde{\beta}_2=\min\{\tilde{\beta}_3,\tilde{\beta}_4\},
	\end{split}
\end{equation*}
where $\tilde{\beta}_3=\min\{N_0+2+\beta_2,\,N_0+2+\alpha-2/p_1\}$ and $\tilde{\beta}_4=\min\{N+1+\beta_1,\,N+1+\alpha-2/p_1,\,N+2-2/p_2\}$.
Combining \eqref{eq:IntDecay} we must have
\begin{equation}\label{eq:IntDecay2}
		\abs{\int_{\weg_{\theta_0,\epa}}\frac{\jump_1}{h}
			V_N(x)\cdot (\nabla\Phi_j) e^{-\Phi_j/h} 
			+ k^2c_2
			v_{N_0} e^{-\Phi_j/h}dx}
		\le Ch^{\tilde{\beta}_2/\alpha},\qquad j =1,2.
\end{equation}
In the following, we split our arguments based on the three different relations between $N_0$ and $N$ as established in Lemma~\ref{lem:TaylorvV}. These three cases are, when $N_0\ge 1$, when $N_0=N=0$, and when $N_0=0$ with $N=1$.
If $N_0\ge 1$, which also corresponds to Item (II-b) in Theorem~\ref{thm:main}, then $N=N_0-1$ and $\tilde{\beta}_2=\tilde{\beta}_4>N+1$. On the other hand, we obtain from Lemmas~\ref{lem:DecayExactv}~and~\ref{lem:DecayExact} that
\begin{equation*}
	\int_{\weg_{\theta_0,\epa}} 
	\frac{\jump_1}{h}
	V_N(x)\cdot (\nabla\Phi_j) e^{-\Phi_j/h} 
	+ k^2c_2
	v_{N_0} e^{-\Phi_j/h}dx
	=-c_1C_{1,j}h^{(N+1)/\alpha} +O(h^{(N+3)/\alpha})
\end{equation*}
for $j=1,2$, where $C_{1,j}$, $j=1,2$, are constants independent of $h$.
Recalling \eqref{eq:IntDecay2}, we obtain that for $v$ admits an extension in $B_{\tilde{\epa}}$, there must hold $C_{1,1}=C_{1,2}=0$, where the latter is equivalent to $2\theta_0={l\pi}/{(1+N)}={l\pi}/{N_0}\in(0,2\pi)\backslash\{\pi\}$ for some integer $l$, by Lemma \ref{lem:DecayExact}.
Otherwise if $N=N_0=0$, which coincides with the case of (II-a) from Theorem~\ref{thm:main}, we have again from Lemmas~\ref{lem:DecayExactv}~and~\ref{lem:DecayExact} that
\begin{equation*}
	\int_{\weg_{\theta_0,\epa}} 
	\frac{\jump_1}{h}
	V_N(x)\cdot (\nabla\Phi_j) e^{-\Phi_j/h} 
	+ k^2c_2
	v_{N_0} e^{-\Phi_j/h}dx
	=-c_1C_{1,j}h^{1/\alpha} +O(h^{2/\alpha}),\quad j=1,2,
\end{equation*}
where $C_{1,j}$, $j=1,2$, are constants independent of $h$, and at least one of them is nonzero (see, Remark~\ref{rem:N=0}). However, this contradicts \eqref{eq:IntDecay2} since $\tilde{\beta}_2=\tilde{\beta}_4>1$, that is, $v$ cannot be extended with $N=N_0=0$.
Lastly, for $N_0=0$ and $N=1$, which matches Item (II-c) in Theorem~\ref{thm:main}, we observe from Lemma~\ref{lem:DecayExact3} that
\begin{equation*}
	\int_{\weg_{\theta_0,\epa}} 
	\frac{\jump_1}{h}
	V_N(x)\cdot (\nabla\Phi_j) e^{-\Phi_j/h} 
	+ k^2c_2
	v_{N_0} e^{-\Phi_j/h}dx
	=-C_{2,j}h^{2/\alpha} +o(e^{-\epa^{\alpha}\delta/(3h)}),\quad j=1,2,
\end{equation*}
where $C_{2,j}$, $j=1,2$, are constants independent of $h$. Since in this case $\tilde{\beta}_2>2$, then \eqref{eq:IntDecay2} yields $C_{2,1}=C_{2,2}=0$, and hence by Lemma~\ref{lem:DecayExact3} that either  $2\theta_0\in\{\pi/2,3\pi/2\}$ and $\jump_1=\jump_2$, or $2\theta_0\notin\{\pi/2,3\pi/2\}$ and
\begin{equation*}
	\nabla v_{2}(x)=V_1(x)=-\frac{1}{4}k^2v_0\nabla ((1-c_0)x_1^2+(1+c_0)x_2^2)
	\qquad\mbox{with}\quad c_0=\frac{1-\jump_2/\jump_1}{\cos2\theta_0}. 
\end{equation*}
For the latter case (when $2\theta_0\notin\{\pi/2,3\pi/2\}$), expanding $v$, which is a solution to $\Delta v+k^2v=0$ in $B_{\tilde{\epa}}$, as the sum of spherical waves $J_{m}(kr)e^{\pm \im m\theta}$, $m\in\mathbb{N}$, we obtain that only the even terms (i.e., when $m$ is even) could be nontrivial. Moreover, the second nonzero term in this expansion is of the form 
\begin{equation*}					
	c_0v_0J_2(kr)(e^{\im 2\theta}+e^{-\im 2\theta}),
\end{equation*}
where the first is $v_0J_{0}(kr)$ with $v_0$ a nonzero constant.

The statements associated with Items (I-b) and (I-c) in Theorem~\ref{thm:main} can be proven by analogous arguments.

\section*{Acknowledgments}
The author would like to thank Professor Fioralba Cakoni for very useful discussions.


\begin{thebibliography}{10}

\bibitem{Ale12UCP}
G.~Alessandrini.
\newblock Strong unique continuation for general elliptic equations in 2{D}.
\newblock {\em J. Math. Anal. Appl.}, 386(2):669--676, 2012.

\bibitem{Bla18}
E.~Bl{\aa}sten.
\newblock Nonradiating sources and transmission eigenfunctions vanish at
  corners and edges.
\newblock {\em SIAM J. Math. Anal.}, 50(6):6255--6270, 2018.

\bibitem{BlaLin19}
E.~Bl{\aa}sten and Y.-H. Lin.
\newblock Radiating and non-radiating sources in elasticity.
\newblock {\em Inverse Problems}, 35(1):015005, 16, 2019.

\bibitem{BlaLiu20arX}
E.~Bl{\aa}sten and H.~Liu.
\newblock Scattering by curvatures, radiationless sources, transmission
  eigenfunctions and inverse scattering problems.
\newblock {\em arXiv preprint, arXiv:1808.01425}, 2020.

\bibitem{BLX19arXiv}
E.~Bl{\aa}sten, H.~Liu, and J.~Xiao.
\newblock On an electromagnetic problem in a corner and its applications.
\newblock {\em Analysis \& PDEs}, 2021.
\newblock To appear.

\bibitem{BPS14}
E.~Bl{\aa}sten, L.~P\"aiv\"arinta, and J.~Sylvester.
\newblock Corners always scatter.
\newblock {\em Communications in Mathematical Physics}, 331(2):725--753, 2014.

\bibitem{BTW20}
E.~Bl{\aa}sten, L.~Tzou, and J.-N. Wang.
\newblock Uniqueness for the inverse boundary value problem with singular
  potentials in 2{D}.
\newblock {\em Math. Z.}, 295(3-4):1521--1535, 2020.

\bibitem{BlaV20}
E.~Bl{\aa}sten and E.~V. Vesalainen.
\newblock Non-scattering energies and transmission eigenvalues in
  {$\mathbb{H}^n$}.
\newblock {\em Ann. Acad. Sci. Fenn. Math.}, 45(1):547--576, 2020.

\bibitem{Bukh08}
A.~L. Bukhgeim.
\newblock Recovering a potential from {C}auchy data in the two-dimensional
  case.
\newblock {\em J. Inverse Ill-Posed Probl.}, 16(1):19--33, 2008.

\bibitem{CCH16}
F.~Cakoni, D.~Colton, and H.~Haddar.
\newblock {\em Inverse scattering theory and transmission eigenvalues},
  volume~88 of {\em CBMS-NSF Regional Conference Series in Applied
  Mathematics}.
\newblock Society for Industrial and Applied Mathematics (SIAM), Philadelphia,
  PA, 2016.

\bibitem{CaX21}
F.~Cakoni and J.~Xiao.
\newblock On corner scattering for operators of divergence form and
  applications to inverse scattering.
\newblock {\em Communications in Partial Differential Equations},
  46(3):413--441, 2021.

\bibitem{ColKre19book}
D.~Colton and R.~Kress.
\newblock {\em Inverse acoustic and electromagnetic scattering theory},
  volume~93 of {\em Applied Mathematical Sciences}.
\newblock Springer, Cham, 2019.
\newblock 4th ed.

\bibitem{ColtonMonk88}
D.~Colton and P.~Monk.
\newblock The inverse scattering problem for time-harmonic acoustic waves in an
  inhomogeneous medium.
\newblock {\em Quart. J. Mech. Appl. Math.}, 41(1):97--125, 1988.

\bibitem{DCH21}
H.~Diao, X.~Cao, and H.~Liu.
\newblock On the geometric structures of transmission eigenfunctions with a
  conductive boundary condition and applications.
\newblock {\em Communications in Partial Differential Equations}, 46(4):1--50,
  2021.

\bibitem{DLMF5_9}
{\it NIST Digital Library of Mathematical Functions (Section 5.9)}.
\newblock http://dlmf.nist.gov/, Release 1.1.0 of 2020-12-15.
\newblock F.~W.~J. Olver, A.~B. {Olde Daalhuis}, D.~W. Lozier, B.~I. Schneider,
  R.~F. Boisvert, C.~W. Clark, B.~R. Miller, B.~V. Saunders, H.~S. Cohl, and
  M.~A. McClain, eds.

\bibitem{ElH18}
J.~Elschner and G.~Hu.
\newblock Acoustic scattering from corners, edges and circular cones.
\newblock {\em Arch. Ration. Mech. Anal.}, 228(2):653--690, 2018.

\bibitem{LiX17Corner}
H.~Liu and J.~Xiao.
\newblock On electromagnetic scattering from a penetrable corner.
\newblock {\em SIAM Journal on Mathematical Analysis}, 49(6):5207--5241, 2017.

\bibitem{PSV17}
L.~P\"aiv\"arinta, M.~Salo, and E.~V. Vesalainen.
\newblock Strictly convex corners scatter.
\newblock {\em Revista Matem\'atica Iberoamericana}, 33(4):1369--1396, 2017.

\bibitem{SyU87}
J.~Sylvester and G.~Uhlmann.
\newblock A global uniqueness theorem for an inverse boundary value problem.
\newblock {\em Annals of Mathematics}, 125(1):153--169, 1987.

\bibitem{Uhl14}
G.~Uhlmann.
\newblock Inverse problems: seeing the unseen.
\newblock {\em Bull. Math. Sci.}, 4(2):209--279, 2014.

\bibitem{Vekua62}
I.~N. Vekua.
\newblock {\em Generalized analytic functions}.
\newblock Pergamon Press, London-Paris-Frankfurt; Addison-Wesley Publishing
  Co., Inc., Reading, Mass., 1962.

\bibitem{VogX20}
M.~S. Vogelius and J.~Xiao.
\newblock Finiteness results concerning non-scattering wave numbers for
  incident plane- and {H}erglotz waves.
\newblock {\em SIAM Journal on Mathematical Analysis}, 2021.
\newblock To appear.

\bibitem{VogX21}
M.~S. Vogelius and J.~Xiao.
\newblock A link between scatterer geometry and finiteness of nonscattering
  wavenumbers for incident herglotz waves.
\newblock 2021.
\newblock In preparation.

\end{thebibliography}

\end{document}